\documentclass{article}

\usepackage{amsmath,amssymb,amsthm,tikz,youngtab,multicol,xcolor,setspace}
\usetikzlibrary{arrows,decorations.pathreplacing,decorations.markings,patterns}
\usepackage{empheq,bbm}
\usepackage{subcaption,float}

\newcommand{\Hom}{\text{Hom}}

\newcommand{\1}{\textbf{1}}
\DeclareMathOperator{\Rep}{Rep}
\DeclareMathOperator{\Vep}{Vec}

\newcommand*{\x}{0.40825}
\newcommand*{\y}{0.5*1.4142}

\theoremstyle{plain}
\newtheorem{theo}{Theorem}
\newtheorem{mytheorem}{Proposition}[subsection]
\newtheorem{lem}[mytheorem]{Lemma}
\newtheorem{cor}[mytheorem]{Corollary}
\theoremstyle{definition}
\newtheorem{example}[mytheorem]{Example}
\newtheorem*{note}{Note}
\newtheorem{defi}[mytheorem]{Definition}

\title{Level bounds for exceptional \\ quantum subgroups in rank two}
\date{\today}
\author{Andrew Schopieray\\ Mathematics Department, University of Oregon}

\begin{document}
\parskip = \baselineskip
\setlength{\parindent}{0cm}

\maketitle

\begin{abstract}
\noindent There is a long-standing belief that the modular tensor categories $\mathcal{C}(\mathfrak{g},k)$, for $k\in\mathbb{Z}_{\geq1}$ and finite-dimensional simple complex Lie algebras $\mathfrak{g}$, contain exceptional connected \'etale algebras at only finitely many levels $k$.  This premise has known implications for the study of relations in the Witt group of nondegenerate braided fusion categories, modular invariants of conformal field theories, and the classification of subfactors in the theory of von Neumann algebras.  Here we confirm this conjecture when $\mathfrak{g}$ has rank 2, contributing proofs and explicit bounds when $\mathfrak{g}$ is of type $B_2$ or $G_2$, adding to the previously known positive results for types $A_1$ and $A_2$.
\end{abstract}

\usetikzlibrary{calc}




\section{Introduction}

\par The moniker \emph{quantum subgroup} has been attached to numerous related concepts; for the purposes of this exposition a quantum subgroup will describe a connected \'etale algebra in a modular tensor category.  Some of the most well-known examples of modular tensor categories arise from the representation theory of quantum groups at roots of unity \cite[Chapter 3]{BaKi} and are of the form $\mathcal{C}(\mathfrak{g},k)$ for some \emph{level} $k\in\mathbb{Z}_{\geq1}$, where $\mathfrak{g}$ is a finite-dimensional complex simple Lie algebra. Theorem \ref{thebigone} states that aside from a predictable infinite family, there are finitely many levels for which a nontrivial connected \'etale algebra can exist in $\mathcal{C}(\mathfrak{g},k)$ for Lie algebras $\mathfrak{g}$ of type $B_2$ and $G_2$, while a complete classification for type $A_2$ due to Gannon is available in the literature under the guise of modular invariants \cite{gannon}.  One should refer to Section \ref{outline} for a statement of Theorem \ref{thebigone} and an outline of the main proof technique.  An explicit level-bound is provided in the subsequent sections, which optimistically allows for a complete classification of connected \'etale algebras in $\mathcal{C}(\mathfrak{g}_2,k)$ and $\mathcal{C}(\mathfrak{so}_5,k)$ by strictly computational methods.

\par One application of Theorem \ref{thebigone} is to partially classify \emph{module categories} over fusion categories $\mathcal{C}:=\mathcal{C}(\mathfrak{g},k)$ where $\mathfrak{g}$ is of type $B_2$ or $G_2$ \cite[Chapter 7]{tcat}.  Each connected \'etale algebra $A\in\mathcal{C}$ gives rise to an indecomposable module category over $\mathcal{C}$ by considering the category of $A$-modules in $\mathcal{C}$, although not all indecomposable module categories can be produced in this way.  For example if $\mathcal{C}$ is a \emph{pointed} modular tensor category \cite[Chapter 8.4]{tcat} with the set of isomorphism classes of simple objects of $\mathcal{C}$ forming a finite abelian group $G$, then indecomposable module categories over $\mathcal{C}$ correspond to subgroups of $G$ along with additional cohomological data \cite[Theorem 3.1]{Ostrikdouble}; this example provides some precedence to title connected \'etale algebras as \emph{quantum subgroups}.  For a non-modular example, module categories over the even parts of the Haagerup subfactors have been classified by Grossman and Snyder \cite{pinhasnoah}.  More classically, module categories over $\mathcal{C}(\mathfrak{sl}_2,k)$ are classified by simply-laced Dynkin diagrams \cite{zuber,KiO} but this characterization scheme has not immediately lent itself to classifying module categories over $\mathcal{C}(\mathfrak{g},k)$ for other simple Lie algebras $\mathfrak{g}$.  The language and tools of tensor categories which have solidified in recent years provide a novel approach to this aging problem.

\par Another reason for seeking a classification of connected \'etale algebras is to find relations in the Witt group of nondegenerate braided fusion categories \cite{DMNO}.   Traditional attempts to classify nondegenerate braided fusion categories include proceeding by rank or by global dimension.  Organizing nondegenerate braided fusion categories by Witt equivalence class offers a powerful albeit indirect approach.  Each Witt equivalence class contains a unique \emph{completely anisotropic} representative $\mathcal{C}$ which is constructed by identifying a maximal connected \'etale algebra $A$ and passing to its category of dyslectic $A$-modules $\mathcal{C}_A^0$.  The only relations that have been completely described are those coming from the subgroup generated by the Witt equivalence classes of pointed categories \cite[Appendix A.7]{DGNO}, $\mathcal{C}(\mathfrak{sl}_2,k)$ for $k\in\mathbb{Z}_{\geq1}$ \cite[Section 5.5]{DNO}, and $\mathcal{C}(\mathfrak{sl}_3,k)$ for $k\in\mathbb{Z}_{\geq1}$ \cite{schopieray2017}.  Classifying connected \'etale algebras in other classes of braided fusion categories is the first step on the path to extending these results.  Witt group relations have also found applications to extensions of vertex operator algebras \cite{huang2015} and anyon condensation \cite{fuchs2013,sebas,kong} in the realm of mathematical physics.

\par Modular tensor categories also encode the data of chiral conformal field theories.  Fuchs, Runkel, and Schweigert \cite{rcft} describe how \emph{full} conformal field theories correspond to the identification of certain commutative algebras in these categories.  These concepts have been recently formalized to \emph{logarithmic} conformal field theories \cite{fuchs2017}, i.e. theories described by non-semisimple analogs of modular tensor categories.  One should also refer to the work of B\"ockenhauer, Evans, and Kawahigashi \cite{kawahigashi1999,kawahigashi2000} which describes this connection in terms of modular invariants and subfactor theory, or Ostrik's  summary of these results in categorical terms \cite[Section 5]{Ostrik2003}.

\par The structure of the paper is as follows: Section \ref{notions} contains the general categorical notions and results needed in the remainder of the exposition, and defines the categories $\mathcal{C}(\mathfrak{g},k)$ when $\mathfrak{g}$ is of rank 2, ending with a geometric discussion of rank 2 fusion rules.  Section \ref{machinery} describes the technical machinery needed to prove Theorem \ref{thebigone} and gives an outline of the main proof as illustrated using $\mathcal{C}(\mathfrak{sl}_2,k)$.  Sections \ref{a2}--\ref{g2} contain the main content of the proof of Theorem \ref{thebigone} for $\mathfrak{sl}_3$, $\mathfrak{so}_5$, and $\mathfrak{g}_2$ respectively, with concluding remarks in Section \ref{conclusion}.




\section{Categorical notions and definitions}\label{notions}


\subsection{Modular tensor categories}\label{sec:defs}

\par All definitions and notations in what follows are standard, and we refer the reader to \cite{tcat} for details as needed.  Our base field will be $\mathbb{C}$, though most of the concepts of this section can be understood over an algebraically closed field of arbitrary characteristic.  

\par \emph{Fusion categories} are $\mathbb{C}$-linear semisimple rigid tensor categories with finitely many isomorphism classes of simple objects, finite dimensional spaces of morphisms and a simple unit object $\textbf{1}$ \cite[Definition 4.1.1]{tcat}.  Fusion categories are abundant in modern mathematics with the most elementary family of examples being $\Rep(G)$, the finite-dimensional complex representations of a finite group $G$.  If $G$ is trivial then this construction produces the trivial fusion category $\Vep$ of finite-dimensional complex vector spaces.

\par For each fusion category $\mathcal{C}$ there exists a unique ring homomorphism $\text{FPdim}:K(\mathcal{C})\to\mathbb{R}$ such that $\text{FPdim}(X)>0$ for any $0\neq X\in\mathcal{C}$, where $K(\mathcal{C})$ is the Grothendieck group of $\mathcal{C}$ given a ring structure via the tensor product of $\mathcal{C}$ \cite[Chapter 4.5]{tcat}.  This \emph{Frobenius-Perron} dimension agrees with the notion of the \emph{categorical} or \emph{quantum} dimension of an object $X\in\mathcal{C}$ in the case $\mathcal{C}$ is a pseudo-unitary fusion category with a properly chosen spherical structure \cite[Definition 9.4.4]{tcat}.  All of the categories considered in the proof of the main theorem will be pseudo-unitary so the dimension of an object $X\in\mathcal{C}$ will simply be denoted $\dim(X)$ in Section \ref{lie} and beyond.

\par A \emph{braiding} on a fusion category $\mathcal{C}$ is a family of natural isomorphisms $c_{X,Y}:X\otimes Y\stackrel{\sim}{\to} Y\otimes X$ for all $X,Y\in\mathcal{C}$ satisfying braid relations \cite[Definition 8.1.1]{tcat}.  Spherical braided fusion categories will be called \emph{pre-modular}.  \emph{Pre}-modular categories may possess simple objects for which the braiding is trivial (with all other objects) making the braiding degenerate in this sense;  specifically, an object $X$ in a braided fusion category is called \emph{transparent} (or central) if $c_{Y,X}\circ c_{X,Y}=\text{id}_{X\otimes Y}$ for all $Y\in\mathcal{C}$.  Modular categories can be seen as those braided fusion categories whose braidings are entirely non-degenerate, or furthest from symmetric as possible.  

\vspace{2 mm}

\begin{defi}
A pre-modular category $\mathcal{C}$ is a \emph{modular tensor category} if $\1\in\mathcal{C}$ is the only simple transparent object.
\end{defi}

\par Lastly we recall that each pre-modular category has a family of natural isomorphisms $\theta(X):X\stackrel{\sim}{\to}X$ for all $X\in\mathcal{C}$ known as \emph{twists} (or ribbon structure), compatible with the inherent braiding isomorphisms \cite[Definition 8.10.1]{tcat}.  We will abuse this notation when it suits our purposes by denoting the complex number $\alpha$ such that $\theta(X)=\alpha\cdot\text{id}_X$ simply as $\theta(X)$ for any simple $X\in\mathcal{C}$.


\subsection{\'Etale algebras}

\par An \emph{algebra} in a fusion category $\mathcal{C}$ is an object $A\in\mathcal{C}$ with multiplication morphism $m:A\otimes A\to A$ and unit morphism $u:\textbf{1}\to A$ satisfying the usual compatibility relations \cite[Definition 7.8.1]{tcat}.  Associativity maps are an inherent structure of the fusion category so in the case $\mathcal{C}=\Vep$, the above definition is equivalent to that of an associative and unital finite-dimensional $\mathbb{C}$-algebra.

\begin{defi}
An algebra $A$ in fusion category $\mathcal{C}$ is \emph{separable} if the multiplication morphism splits as a map of $A$-bimodules and \emph{connected \'etale} if $\Hom_\mathcal{C}(\1,A)=1$ (connected) and $A$ is commutative and separable (\'etale).
\end{defi}
As described in \cite[Section 3]{DMNO}, the condition that an algebra $A$ is separable is equivalent to the category of right $A$-modules $\mathcal{C}_A$ being semisimple.  Furthermore $A$ connected \'etale implies $C_A$ is a fusion category and moreover if $\mathcal{C}$ is braided then $\mathcal{C}_A^0$, the category of \emph{dyslectic} right $A$-modules is braided.

\begin{example}\label{twotwotwo}
A fusion category $\mathcal{C}$ has at least one connected \'etale algebra: the unit object $\1$ whose multiplication $\1\otimes\1\to\1$ and unit maps $\1\to\1$ are part of the monoidal data of $\mathcal{C}$.  Isomorphism classes of simple $A$-modules are then in one-to-one correspondence with isomorphism classes of simple objects of $\mathcal{C}$ with the unit morphisms of $\mathcal{C}$ acting as the $A$-module structure maps.  Less trivially, the algebra of complex functions on a finite group $G$ has a structure of a connected \'etale algebra in $\Rep(G)$ by which $G$ acts on complex functions by right translations.  Refer to \cite[Example 7.8.3]{tcat} for additional nontrivial examples.
\end{example}

\par The numerical conditions for an algebra in a pseudo-unitary pre-modular category to be connected \'etale are quite restrictive.  In particular the full twist on such an algebra is trivial as we will prove below.  This result is due to Victor Ostrik, although a proof does not appear in the literature to our knowledge.  The full twist need not be trivial if the assumption of pseudo-unitary is removed as the following example illustrates.

\begin{example}\label{svep}
The fusion category of complex $\mathbb{Z}/2\mathbb{Z}$-graded vector spaces has two possible (symmetric) pre-modular structures, distinguished by the full twist on the non-trivial simple object $\theta(X)=\pm1$.  The trivial twist corresponds to the pseudo-unitary category Rep$(\mathbb{Z}/2\mathbb{Z})$, while the nontrivial twist corresponds to s$\Vep$, the category of complex super vector spaces \cite[Example 8.2.2]{tcat}.  The object $A:=\textbf{1}\oplus X$ has a unique structure of a connected \'etale algebra in both cases, but $\theta(A)\neq\text{id}_A$ in s$\Vep$, which is \emph{not} pseudo-unitary (i.e. $\dim(X)=-1$).
\end{example}

\par The main concept behind the proof of Lemma \ref{twistisone} is to reduce the argument to the cases in Example \ref{svep}.

\begin{lem}\label{twistisone}
If $\mathcal{C}$ is a pseudo-unitary braided fusion category and $A$ is a connected \'etale algebra in $\mathcal{C}$, then $\theta(A)=\text{id}_A$.
\end{lem}

\begin{proof}
The composition $\varphi:A\otimes A\stackrel{m}{\to}A\stackrel{\varepsilon_A}{\to}\textbf{1}$ is non-degenerate \cite[Remark 3.4]{DMNO}, where $\varepsilon_A$ arises from $A$ being connected (and is unique up to scalar multiple).  Note that the commutativity of $A$ implies $\varphi s_{X^\ast,X}s_{X,X^\ast}=\varphi$.  We can then rewrite $s_{X^\ast,X}s_{X,X^\ast}$ using the balancing axiom \cite[Equation 2.2.8]{BaKi} to yield $\theta(X)\theta(X^\ast)\theta(\textbf{1})^{-1}=1$ because $\varphi$ is nondegenerate.  Moreover $\theta(X)=\pm1$.  So we may now decompose $A=A^+\oplus A^-$ where $A^{\pm}$ is the sum of simple summands of $A$ with twist $\pm1$, respectively.  We will deduce that $A^-$ is empty in the remainder of the proof.

\par The commutativity of $A=A^+\oplus A^-$ implies this decomposition is a $\mathbb{Z}/2\mathbb{Z}$-grading again by the balancing axiom, i.e. $\theta(X\otimes Y)=\theta(X)\theta(Y)$ for all simple $X,Y\subset A$.  Thus $m$ restricts to a multiplication morphism $A^+\otimes A^+\to A^+$.  We now aim to prove that $A^+$ is a connected \'etale algebra.  The commutativity of $A^+$ is clear from the commutativity of $A$ and $\theta(\textbf{1})=1$ by the balancing axiom \cite[Equation 2.2.9]{BaKi} so $A^+$ is connected.  It remains to show that $A^+$ is separable, i.e. $\mathcal{C}_{A^+}$ is semisimple.  This follows from \cite[Theorem 3.3]{KiO} by recalling that $A^+$ is \emph{rigid} (in the sense of Kirillov and Ostrik) because $A^+\otimes A^+\stackrel{m}{\to}A^+\stackrel{\varepsilon_A}{\to}\textbf{1}$ is non-degenerate, and $\dim(A)\neq0$ since $\mathcal{C}$ is pseudo-unitary.

\par In the language of \cite[Section 3.6]{DMNO}, $A$ with the inclusion $A^+\to A$ is known as a commutative algebra \emph{over} $A^+$ and thus $A$ can be considered as a commutative algebra in $\mathcal{D}:=\mathcal{C}_{A^+}^0$.  Proposition 3.16 of \cite{DMNO} then implies $A$ (as an algebra in $\mathcal{D}$) is connected \'etale as well.  We also note that $\theta(A^+)=\text{id}_{A^+}$ along with Theorem 1.18 of \cite{KiO}, implies $\dim(\mathcal{D})=\sum_{X\in\mathcal{O}(\mathcal{D})}\dim_\mathcal{D}(X)^2$ is equal to
\begin{align}
\sum_{X\in\mathcal{O}(\mathcal{D})}\left(\dfrac{\dim_\mathcal{C}(X)}{\dim_\mathcal{C}(A)}\right)^2&=\text{FPdim}_\mathcal{C}(A)^{-2}\sum_{X\in\mathcal{O}(\mathcal{D})}\text{FPdim}_\mathcal{C}(X)^2 \label{twistone}\\
&=\text{FPdim}_\mathcal{C}(A)^{-2}\text{FPdim}(\mathcal{C}) \nonumber\\
&=\text{FPdim}(\mathcal{D}).\label{twisttwo}
\end{align}
where (\ref{twistone}) follows from $\mathcal{C}$ being pseudo-unitary and \cite[Corollary 3.32]{DMNO} implies (\ref{twisttwo}).  Moreover we have shown $\mathcal{D}$ is pseudo-unitary by Proposition 8.23 of \cite{ENO}.

\par Now assume $X\subset A$ is a simple summand of $A$ (as an object of $\mathcal{D}$) which is \emph{distinct} from $A^+=\textbf{1}_\mathcal{D}$.  An immediate consequence is that $X\subset A^-$, should such an object exist.  But $X\otimes_\mathcal{D}X$ is a quotient of $X\otimes X$ \cite[Theorem 1.5]{KiO} which, by the $\mathbb{Z}/2\mathbb{Z}$-grading of $A$, implies $X\otimes_\mathcal{D}X\subset A^+=\textbf{1}_\mathcal{D}$.  The simplicity of the unit object $\textbf{1}_\mathcal{D}=A^+$ then implies $X\otimes_\mathcal{D}X=1_\mathcal{D}$.

\par Lastly we consider the fusion subcategory $\mathcal{E}\subset\mathcal{D}$ generated by $X$, which by the above reasoning is equivalent to the category of $\mathbb{Z}/2\mathbb{Z}$-graded vector spaces (as a fusion category).  The spherical structure of $\mathcal{E}$ which is inherited from $\mathcal{D}$, must be the nontrivial one since $\theta(X)=-1$ (see Example \ref{svep}) and thus $\mathcal{E}=\text{s}\Vep$, a contradiction to $\mathcal{D}$ being pseudo-unitary.  Moreover no such $X$ can exist and $A=A^+$.
\end{proof}


\subsection{The categories $\mathcal{C}(\mathfrak{g},k)$}\label{lie}

\par Here we organize the requisite numerical data of our modular tensor categories of interest, the categories $\mathcal{C}(\mathfrak{g},k)$ where $\mathfrak{g}$ is, for instructive purposes, $\mathfrak{sl}_2$ or $\mathfrak{sl}_3$, and, for the main result of this paper, $\mathfrak{so}_5$ or $\mathfrak{g}_2$.

\par If $\mathfrak{g}$ is a finite-dimensional simple Lie algebra and $\hat{\mathfrak{g}}$ is the corresponding affine Lie algebra, then for all $k\in\mathbb{Z}_{>0}$ one can associate a pseudo-unitary modular tensor category $\mathcal{C}(\mathfrak{g},k)$ consisting of highest weight integrable $\hat{\mathfrak{g}}$-modules of level $k$.  These categories were studied by Andersen and Paradowski \cite{ap} and Finkelberg \cite{fink} later proved that $\mathcal{C}(\mathfrak{g},k)$ is equivalent to the semisimple portion of the representation category of Lusztig's quantum group $\mathcal{U}_q(\mathfrak{g})$ when $q=e^{\pi i/(k+h^\vee)}$ (Figure \ref{fig:q}) where $h^\vee$ is the dual coxeter number for $\mathfrak{g}$ \cite[Chapter 7]{BaKi}.
\begin{figure}[H]
\centering
\begin{equation*}
\begin{array}{|c|c|}
\hline  \mathfrak{g}  & q\\ \hline
\mathfrak{sl}_2 & \exp(\pi i/(k+2))  \\
\mathfrak{sl}_3 & \exp(\pi i/(k+3)) \\
\mathfrak{so}_5 & \exp((1/2)\pi i/(k+3)) \\
\mathfrak{g}_2 & \exp((1/3)\pi i/(k+4)) \\\hline
\end{array}
\end{equation*}
    \caption{Roots of unity $q$}%
    \label{fig:q}%
\end{figure}
\par Let $\mathfrak{h}$ be a Cartan subalgebra of $\mathfrak{g}$ and $\langle.\,,.\rangle$ be the invariant form on $\mathfrak{h}^\ast$ normalized so that $\langle\alpha,\alpha\rangle=2$ for short roots \cite[Section 5]{hump}.  Simple objects of $\mathcal{C}(\mathfrak{g},k)$ are labelled by weights $\lambda\in\Lambda_0\subset\mathfrak{h}^\ast$, the \emph{Weyl alcove} of $\mathfrak{g}$ at level $k$.  Simple objects and their representative weights will be used interchangably but can be easily determined by context.  Geometrically, $\Lambda_0$ can be described as those weights bounded by the \emph{walls} of $\Lambda_0$: the hyperplanes $T_i:=\{\lambda\in\mathfrak{h}^\ast:\langle \lambda+\rho,\alpha_i\rangle=0\}$ for each simple root $\alpha_i\in\mathfrak{h}^\ast$ and $T_0:=\{\lambda\in\mathfrak{h}^\ast:\langle\lambda+\rho,\theta^\vee\rangle<k+h^\vee\}$ where $\theta$ is the longest dominant root.  Reflections through the hyperplane $T_i$ will be denoted $\tau_i$ which generate the \emph{affine Weyl group} $\mathfrak{W}_0$.

\par If $\rho$ is the half-sum of all positive roots of $\mathfrak{g}$ then the dimension of the simple object corresponding to the weight $\lambda\in\Lambda_0$ is given by the \emph{quantum Weyl dimension formula}
\begin{equation*}\dim(\lambda)=\prod_{\alpha\succ0}\dfrac{[\langle\alpha,\lambda+\rho\rangle]}{[\langle\alpha,\rho\rangle]}\end{equation*}
where $[m]$ is the \emph{$q$-analog} of $m\in\mathbb{Z}_{\geq0}$ which for a generic parameter $q$ is
\begin{equation*}[m]=\dfrac{q^m-q^{-m}}{q-q^{-1}}.\end{equation*}
In what follows the numerator of the quantum Weyl dimension formula will be all that needs to be considered as only equalities and inequalities of dimensions with equal denominators appear.  We will denote this numerator $\dim'(\lambda)$.   With the values of $q$ found in Figure \ref{fig:q}, $\dim(\lambda)\in\mathbb{R}_{\geq1}$ (and in particular $[m]\in\mathbb{R}_{>0}$ for all considered $m\in\mathbb{Z}_{>0}$) for all $\lambda\in\Lambda_0$.  The full twist on a simple object $\lambda\in\Lambda_0$ is given by $\theta(\lambda)=q^{\langle\lambda,\lambda+2\rho\rangle}$ which is a root of unity depending on $\mathfrak{g}$, $k$, and $\lambda$.

\par The fusion rules for the categories $\mathcal{C}(\mathfrak{g},k)$ are given by the \emph{quantum Racah formula} \cite[Corollary 8]{Sawin03}.  If $\lambda,\gamma,\mu\in\Lambda_0$, the multiplicity of $\mu$ in the product $\lambda\otimes\gamma$ is the fusion coefficient
\begin{equation*}N_{\lambda,\gamma}^\mu:=\sum_{\tau\in\mathfrak{W}_0}(-1)^{\ell(\tau)}m_\lambda(\tau(\mu)-\gamma)\end{equation*}
where $m_\lambda(\mu)$ is the (classical) dimension of the $\mu$-weight space of the finite dimensional irreducible representation of highest weight $\lambda$.  We refer the reader to \cite[Sections 13,21--24]{hump} for concepts and results from classical representation theory of Lie algebras.


\subsubsection{$\mathcal{C}(\mathfrak{sl}_2,k)$}\label{sec:a1}

\par Simple objects of $\mathcal{C}(\mathfrak{sl}_2,k)$ are enumerated by $s\in\mathbb{Z}_{\geq0}$ such that $s\leq k$.  Each object, denoted by $(s)$, corresponds to the weight $s\lambda\in\Lambda_0$, where $\lambda$ is the unique fundamental weight.  The dimension of $(s)$ is given by $\dim(s)=[s+1]$ and the full twist on this object by
\begin{equation*}\theta(s)=\exp\left(\dfrac{s(s+2)}{4(k+2)}\cdot2\pi i\right).\end{equation*}
Figures \ref{fig:a1}--\ref{fig:g2} contain geometric visualizations of the Weyl alcove with respect to the specified Lie algebra and level, with nodes representing weights in $\Lambda_0$ and shaded nodes representing those weights which also lie in the root lattice.  Walls of $\Lambda_0$ are illustrated by dashed lines.

\begin{figure}[H]
\centering
\begin{tikzpicture}
\draw[fill=black!25] (0,0) node[above] {\tiny$(0)$} circle (0.1);
\draw[fill=white] (1,0)node[above] {\tiny$\lambda$} circle (0.1);
\draw[fill=black!25] (2,0) circle (0.1);
\draw[fill=white] (3,0)  circle (0.1);
\draw[fill=black!25] (4,0)  circle (0.1);
\draw[fill=white] (5,0)  circle (0.1);
\draw[fill=black!25] (6,0) node[above] {\tiny$(6)$}  circle (0.1);
\draw[dashed] (7,0.5) --node[circle,draw=black,pos = 0.5,inner sep=1pt,fill=white,solid,thin] {\tiny$T_0$} (7,-0.5);
\draw[dashed] (-1,0.5) --node[circle,draw=black,pos = 0.5,inner sep=1pt,fill=white,solid,thin] {\tiny$T_1$} (-1,-0.5);
\end{tikzpicture}
    \caption{$\mathcal{C}(\mathfrak{sl}_2,6)$}%
    \label{fig:a1}%
\end{figure}
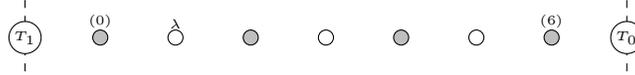


\subsubsection{$\mathcal{C}(\mathfrak{sl}_3,k)$}\label{sec:a2}

\par Simple objects of $\mathcal{C}(\mathfrak{sl}_3,k)$ are enumerated by nonnegative integer pairs $(s,t)$, such that $s+t\leq k$.  Each $(s,t)$ corresponds to the weight $s\lambda_1+t\lambda_2\in\Lambda_0$.  The dimension of the simple object $(s,t)$ is given by
\begin{equation*}\dim(s,t)=\dfrac{1}{[2]}[s+1][t+1][s+t+2],\end{equation*}
and the twist on this object by
\begin{equation*}\theta(s,t)=\exp\left(\frac{s^2+3s+st+3t+t^2}{3(k+3)}\cdot2\pi i\right).\end{equation*}
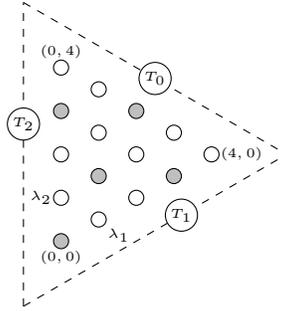
\begin{figure}[H]
\centering
\begin{tikzpicture}
\draw[fill=black!25] (0,0) circle (0.1) node[below] {\tiny$(0,0)$};
\draw[fill=white!20] (0,1.732/3)  node[left] {\tiny$\lambda_2$} circle (0.1);
\draw[fill=white!20] (0,2*1.732/3) circle (0.1);
\draw[fill=black!25] (0,3*1.732/3) circle (0.1);
\draw[fill=white!20] (0,4*1.732/3) circle (0.1) node[above] {\tiny$(0,4)$};
\draw[fill=white!20] (1/2,0.5*1.732/3) node[below right] {\tiny$\lambda_1$} circle (0.1);
\draw[fill=black!25] (1/2,1*1.732/3+0.5*1.732/3) circle (0.1);
\draw[fill=white!20] (1/2,2*1.732/3+0.5*1.732/3) circle (0.1);
\draw[fill=white!20] (1/2,3*1.732/3+0.5*1.732/3) circle (0.1);
\draw[fill=white!20] (1,1*1.732/3) circle (0.1);
\draw[fill=white!20] (1,2*1.732/3) circle (0.1);
\draw[fill=black!25] (1,3*1.732/3) circle (0.1);
\draw[fill=black!25] (3/2,1.5*1.732/3) circle (0.1);
\draw[fill=white!20] (3/2,2.5*1.732/3) circle (0.1);
\draw[fill=white!20] (2,2*1.732/3) circle (0.1) node[right] {\tiny$(4,0)$};
\draw[dashed] (-0.5,-1.5*1.732/3) --node[circle,draw=black,pos = 0.6,inner sep=1pt,fill=white,solid,thin] {\tiny$T_1$} (3,2*1.732/3);
\draw[dashed] (-0.5,-1.5*1.732/3) --node[circle,draw=black,pos = 0.6,inner sep=1pt,fill=white,solid,thin] {\tiny$T_2$} (-0.5,5.5*1.732/3);
\draw[dashed] (-0.5,5.5*1.732/3) --node[circle,draw=black,pos = 0.5,inner sep=1pt,fill=white,solid,thin] {\tiny$T_0$} (3,2*1.732/3);
\end{tikzpicture}
    \caption{$\mathcal{C}(\mathfrak{sl}_3,4)$}%
    \label{fig:a2}%
\end{figure}


\subsubsection{$\mathcal{C}(\mathfrak{so}_5,k)$}\label{sec:so5}

\par Simple objects of $\mathcal{C}(\mathfrak{so}_5,k)$ are enumerated by nonnegative integer pairs $(s,t)$, such that $s+t\leq k$.  Each $(s,t)$ corresponds to the weight $s\lambda_1+t\lambda_2\in\Lambda_0$.  The dimension of the simple object of $\mathcal{C}(\mathfrak{so}_5,k)$ corresponding to the weight $(s,t)$ is given by
\begin{equation*}\dim(s,t)=\dfrac{[2(s+1)][t+1][2(s+t+2)][2s+t+3]}{[2][3][4][1]},\label{so5dim}\end{equation*}
and the twist on this object by
\begin{equation*}\theta(s,t)=\exp\left(\frac{2s^2+2st+6s+t^2+4t}{4(k+3)}\cdot2\pi i\right).\label{so5twist}\end{equation*}
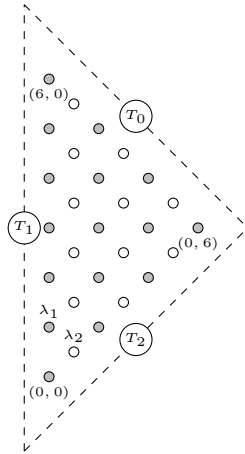
\begin{figure}[H]
\centering
\begin{tikzpicture}[scale=0.66]
\draw[fill=black!25] (0,0) circle (0.1) node[below] {\tiny$(0,0)$};
\draw[fill=black!25] (0,1) circle (0.1) node[above] {\tiny$\lambda_1$};
\draw[fill=black!25] (0,2) circle (0.1);
\draw[fill=black!25] (0,3) circle (0.1);
\draw[fill=black!25] (0,4) circle (0.1);
\draw[fill=black!25] (0,5) circle (0.1);
\draw[fill=black!25] (0,6) circle (0.1) node[below] {\tiny$(6,0)$};
\draw[fill=white!20] (1/2,0+1/2) circle (0.1) node[above] {\tiny$\lambda_2$};
\draw[fill=white!20] (1/2,1+1/2) circle (0.1);
\draw[fill=white!20] (1/2,2+1/2) circle (0.1);
\draw[fill=white!20] (1/2,3+1/2) circle (0.1);
\draw[fill=white!20] (1/2,4+1/2) circle (0.1);
\draw[fill=white!20] (1/2,5+1/2) circle (0.1);
\draw[fill=black!25] (1,1) circle (0.1);
\draw[fill=black!25] (1,2) circle (0.1);
\draw[fill=black!25] (1,3) circle (0.1);
\draw[fill=black!25] (1,4) circle (0.1);
\draw[fill=black!25] (1,5) circle (0.1);
\draw[fill=white!20] (3/2,1+1/2) circle (0.1);
\draw[fill=white!20] (3/2,2+1/2) circle (0.1);
\draw[fill=white!20] (3/2,3+1/2) circle (0.1);
\draw[fill=white!20] (3/2,4+1/2) circle (0.1);
\draw[fill=black!25] (2,2) circle (0.1);
\draw[fill=black!25] (2,3) circle (0.1);
\draw[fill=black!25] (2,4) circle (0.1);
\draw[fill=white!20] (5/2,2+1/2) circle (0.1);
\draw[fill=white!20] (5/2,3+1/2) circle (0.1);
\draw[fill=black!25] (3,3) circle (0.1) node[below] {\tiny$(0,6)$};
\draw[dashed] (-0.5,-1.5) -- node[circle,draw=black,pos = 0.5,inner sep=1pt,fill=white,solid,thin] {\tiny$T_1$} (-0.5,7.5);
\draw[dashed] (-0.5,-1.5) -- node[circle,draw=black,pos = 0.5,inner sep=1pt,fill=white,solid,thin] {\tiny$T_2$} (4,3);
\draw[dashed] (-0.5,7.5) -- node[circle,draw=black,pos = 0.5,inner sep=1pt,fill=white,solid,thin] {\tiny$T_0$} (4,3);
\end{tikzpicture}
    \caption{$\mathcal{C}(\mathfrak{so}_5,6)$}%
    \label{fig:so5}%
\end{figure}


\subsubsection{$\mathcal{C}(\mathfrak{g}_2,k)$}

\par Simple objects of $\mathcal{C}(\mathfrak{g}_2,k)$ are enumerated by nonnegative integer pairs $(s,t)$, such that $s+2t\leq k$.  Each $(s,t)$ corresponds to the weight $s\lambda_1+t\lambda_2\in\Lambda_0$.  The dimension of the simple object $(s,t)$ is given by
\begin{equation*}
\dim(s,t)=\dfrac{[s+1][3(t+1)][3(s+t+2)][3(s+2t+3)][s+3t+4][2s+3t+5]}{[1][3][6][9][4][5]},
\end{equation*}
and the twist on this object by
\begin{equation*}\theta(s,t)=\exp\left(\dfrac{s^2+3st+5s+3t^2+9t}{3(k+4)}\cdot2\pi i\right).\end{equation*}
\begin{figure}[H]
\centering
\begin{tikzpicture}[scale=0.33]
\draw[fill=black!25] (0*0.5*1.414,0*1.225+0*2.4495) circle (0.15) node[below,xshift = 0.05cm] {\tiny$(0,\!0)$};
\draw[fill=black!25] (0*0.5*1.414,0*1.225+1*2.4495) circle (0.15) node[below] {\tiny$\lambda_2$};
\draw[fill=black!25] (1*0.5*1.414,1*1.225+0*2.4495) circle (0.15) node[right] {\tiny$\lambda_1$};
\draw[fill=black!25] (8*0.5*1.414,8*1.225+0*2.4495) circle (0.15) node[above] {\tiny$(8,\!0)$};
\draw[fill=black!25] (7*0.5*1.414,7*1.225+0*2.4495) circle (0.15);
\draw[fill=black!25] (6*0.5*1.414,6*1.225+0*2.4495) circle (0.15);
\draw[fill=black!25] (5*0.5*1.414,5*1.225+0*2.4495) circle (0.15);
\draw[fill=black!25] (4*0.5*1.414,4*1.225+0*2.4495) circle (0.15);
\draw[fill=black!25] (3*0.5*1.414,3*1.225+0*2.4495) circle (0.15);
\draw[fill=black!25] (2*0.5*1.414,2*1.225+0*2.4495) circle (0.15);
\draw[fill=black!25] (0*0.5*1.414,0*1.225+2*2.4495) circle (0.15);
\draw[fill=black!25] (0*0.5*1.414,0*1.225+3*2.4495) circle (0.15);
\draw[fill=black!25] (0*0.5*1.414,0*1.225+4*2.4495) circle (0.15) node[above,xshift = 0.05cm] {\tiny$(0,\!4)$};
\draw[fill=black!25] (1*0.5*1.414,1*1.225+1*2.4495) circle (0.15); 
\draw[fill=black!25] (2*0.5*1.414,2*1.225+1*2.4495) circle (0.15);
\draw[fill=black!25] (1*0.5*1.414,1*1.225+2*2.4495) circle (0.15);
\draw[fill=black!25] (1*0.5*1.414,1*1.225+2*2.4495)  circle (0.15);
\draw[fill=black!25] (1*0.5*1.414,1*1.225+3*2.4495) circle (0.15);
\draw[fill=black!25] (2*0.5*1.414,2*1.225+2*2.4495) circle (0.15);
\draw[fill=black!25] (2*0.5*1.414,2*1.225+3*2.4495) circle (0.15);
\draw[fill=black!25]  (3*0.5*1.414,3*1.225+1*2.4495) circle (0.15);
\draw[fill=black!25] (3*0.5*1.414,3*1.225+2*2.4495) circle (0.15);
\draw[fill=black!25] (4*0.5*1.414,4*1.225+1*2.4495) circle (0.15);
\draw[fill=black!25] (4*0.5*1.414,4*1.225+2*2.4495) circle (0.15);
\draw[fill=black!25] (5*0.5*1.414,5*1.225+1*2.4495) circle (0.15);
\draw[fill=black!25] (6*0.5*1.414,6*1.225+1*2.4495) circle (0.15);
\draw[dashed] (-1*0.5*1.414,-1*1.225-1*2.4495) --node[circle,draw=black,pos = 0.5,inner sep=1pt,fill=white,solid,thin] {\tiny$T_2$} (-1*0.5*1.414,-1*1.225+5*2.4495);
\draw[dashed] (-1*0.5*1.414,-1*1.225-1*2.4495) --node[circle,draw=black,pos = 0.5,inner sep=1pt,fill=white,solid,thin] {\tiny$T_1$} (11*0.5*1.414,11*1.225-1*2.4495);
\draw[dashed] (-1*0.5*1.414,-1*1.225+5*2.4495) --node[circle,draw=black,pos = 0.5,inner sep=1pt,fill=white,solid,thin] {\tiny$T_0$} (11*0.5*1.414,11*1.225-1*2.4495);
\end{tikzpicture}
    \caption{$\mathcal{C}(\mathfrak{g}_2,8)$}%
    \label{fig:g2}%
\end{figure}
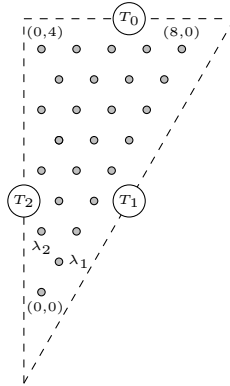


\subsection{Fusion rules in rank 2}\label{sec:geometry}

\par It is necessary to the proof of future claims to consider the geometric interpretation of the quantum Racah formula specifically for rank 2 Lie algebras \cite[Remark 4]{Sawin03}.  The notation and concepts introduced in this subsection will be used prolifically throughout the proof of Theorem \ref{thebigone} and are illustrated by example in Figure \ref{fig:b2example} to compute $N_{\lambda,\gamma}^\mu$ for arbitrary $\mu\in\Lambda_0$, $\lambda:=(3,4)$, and $\gamma:=(3,6)$ (white node) in $\mathcal{C}(\mathfrak{so}_5,12)$.

\par Given $\lambda,\gamma\in\Lambda_0$, the quantum Racah formula states that to calculate the fusion coefficients $N_{\lambda,\gamma}^\mu$ for any $\mu\in\Lambda_0$ geometrically, one should compute $\Pi(\lambda)$, the classical weight diagram for the finite-dimensional irreducible representation of highest weight $\lambda$, and (for visual ease) we illustrate its convex hull, $\overline{\Pi}(\lambda)$.  For this purpose reflections in the classical Weyl group are illustrated in Figure \ref{fig:2sub1} by thin lines.  One can then shift $\overline{\Pi}(\lambda)$ and $\Pi(\lambda)$ so they are centered at $\gamma$, denoting these shifted sets by $\overline{\Pi}(\lambda:\gamma)$ and $\Pi(\lambda:\gamma)$.  Now for a fixed weight $\mu\in\Lambda_0$, $\tau\in\mathfrak{W}_0$ will contribute to the sum $N_{\lambda,\gamma}^\mu$ if and only if there exists $\mu'\in\Pi(\lambda:\gamma)$ such that $\tau(\mu')=\mu$.  The walls of $\Lambda_0$ are illustrated (and labelled) in Figure \ref{fig:2sub2} by dashed lines and all contributing $\tau\in\mathfrak{W}_0$ can be visualized by \emph{folding} $\overline{\Pi}(\lambda:\gamma)$ along the walls of $\Lambda_0$ until it lies completely within $\Lambda_0$.  To emphasize the effect of folding, the folded portions of $\overline{\Pi}(\lambda:\gamma)$ are illustrated in Figure \ref{fig:2sub2} with emphasized shading, while regions of $\overline{\Pi}(\lambda:\gamma)$ unaffected by folding are given a crosshatch pattern.

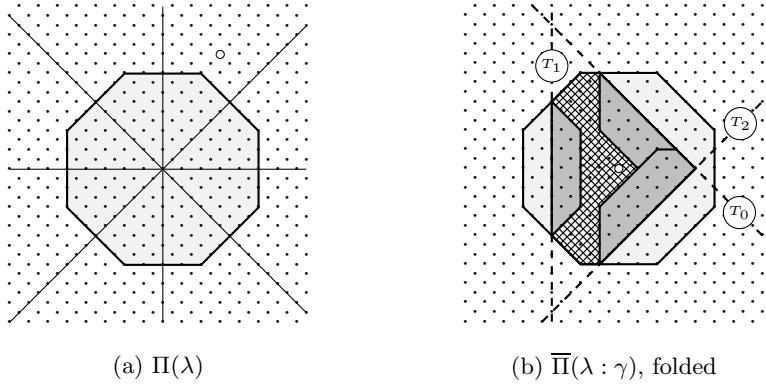
\begin{figure}[H]
\centering
\begin{subfigure}{.5\textwidth}
  \centering
\begin{tikzpicture}[scale=0.18]
\draw[fill=black!5,thick] (2*1.414,5*1.414) -- (5*1.414,2*1.414) -- (5*1.414,-2*1.414) -- (2*1.414,-5*1.414) -- (-2*1.414,-5*1.414) -- (-5*1.414,-2*1.414) -- (-5*1.414,2*1.414) -- (-2*1.414,5*1.414) -- cycle;
\foreach \x in {-8,-7,...,7} {
	\foreach \y in {-8,-7,...,8} {
        \node at (\x*1.4142,\y*1.4142) {$\cdot$};
    		}
		};
\foreach \x in {-8,-7,...,7} {
	\foreach \y in {-8,-7,...,8} {
        \node at (\x*1.4142+0.707,\y*1.4142+0.707) {$\cdot$};
    		}
		};
\draw[thin] (0*1.414,-8*1.414) -- (0*1.414,8.5*1.414);
\draw[thin] (-8*1.414,0*1.414) -- (7.5*1.414,0*1.414);
\draw[thin] (-8*1.414,8*1.414) -- (7.5*1.414,-7.5*1.414);
\draw[thin] (-8*1.414,-8*1.414) -- (7.5*1.414,7.5*1.414);
\draw[fill=white] (3*1.414,6*1.414) circle (0.3);
\end{tikzpicture}
  \caption{$\Pi(\lambda)$}
  \label{fig:2sub1}
\end{subfigure}%
\begin{subfigure}{.5\textwidth}
  \centering
\begin{tikzpicture}[scale=0.18]
\draw[fill=black!5,thick] (2*1.414,5*1.414) -- (5*1.414,2*1.414) -- (5*1.414,-2*1.414) -- (2*1.414,-5*1.414) -- (-2*1.414,-5*1.414) -- (-5*1.414,-2*1.414) -- (-5*1.414,2*1.414) -- (-2*1.414,5*1.414) -- cycle;
\draw[pattern=crosshatch,thick]  (-3.5*1.414,3.5*1.414)--   (-2*1.414,5*1.414)  --   (-1*1.414,5*1.414) -- (4*1.414,0*1.414) -- (-1*1.414,-5*1.414) -- (-2*1.414,-5*1.414) -- (-3.5*1.414,-3.5*1.414) -- cycle;
\draw[fill=black!25,thick] (-3.5*1.414,3.5*1.414) -- (-2*1.414,2*1.414)  --(-2*1.414,-2*1.414)-- (-3.5*1.414,-3.5*1.414) -- cycle;
\draw[fill=black!25,thick] (-1*1.414,2*1.414) -- (-1*1.414,5*1.414) -- (4*1.414,0*1.414) -- (2.5*1.414,-1.5*1.414) -- cycle;
\draw[fill=black!25,thick] (-1*1.414,-2*1.414) -- (2*1.414,1*1.414) -- (3*1.414,1*1.414)  -- (4*1.414,0*1.414) --(-1*1.414,-5*1.414)-- cycle;
\foreach \x in {-8,-7,...,7} {
	\foreach \y in {-8,-7,...,8} {
        \node at (\x*1.4142,\y*1.4142) {$\cdot$};
    		}
		};
\foreach \x in {-8,-7,...,7} {
	\foreach \y in {-8,-7,...,8} {
        \node at (\x*1.4142+0.707,\y*1.4142+0.707) {$\cdot$};
    		}
		};
\draw[dashed,thick] (-3.5*1.414,-8*1.414) -- node[circle,draw=black,pos = 0.81,inner sep=1pt,fill=white,solid,thin] {\tiny$T_1$}  (-3.5*1.414,8.5*1.414);
\draw[dashed,thick] (-4*1.414,-8*1.414) -- node[circle,draw=black,pos = 0.9,inner sep=1pt,fill=white,solid,thin] {\tiny$T_2$}   (7.5*1.414,3.5*1.414);
\draw[dashed,thick] (7.5*1.414,-3.5*1.414) -- node[circle,draw=black,pos = 0.1,inner sep=1pt,fill=white,solid,thin] {\tiny$T_0$}   (-4.5*1.414,8.5*1.414);
\draw[fill=white] (0*1.414,0*1.414) circle (0.3);
\end{tikzpicture}
  \caption{$\overline{\Pi}(\lambda:\gamma)$, folded}
  \label{fig:2sub2}
\end{subfigure}
\caption{$\lambda\otimes\gamma\in\mathcal{C}(\mathfrak{so}_5,12)$}
\label{fig:b2example}
\end{figure}

\par For arbitrary $\lambda,\gamma,\mu\in\Lambda_0$ there may be several $\tau\in\mathfrak{W}_0$ which contribute (positively or negatively) to the sum $N_{\lambda,\gamma}^\mu$ in the quantum Racah formula, but for many fusion coefficients the only contribution comes from the identity of $\mathfrak{W}_0$ and are therefore easily determined to be zero or positive.  In Figure \ref{fig:2sub2}, these coefficients correspond to weights in both $\Pi(\lambda:\gamma)$ and the crosshatched region.

\begin{lem}\label{fusion}
Fix $\lambda,\gamma,\mu\in\Lambda_0$. If
\begin{enumerate}
\item[(1)] $\mu\in\Pi(\lambda:\gamma)$, and
\item[(2)] $\tau_i(\mu')\neq\mu$ for any $\mu'\in\Pi(\lambda:\gamma)$ and $i=0,1,2$,
\end{enumerate}
then $N_{\lambda,\gamma}^\mu>0$.
\end{lem}

\begin{proof}
By assumption (1), $m_\lambda(\mu-\gamma)>0$ is one term in the quantum Racah formula for $N_{\lambda,\gamma}^\mu$.  Any nontrivial $\tau$ contributing to $N_{\lambda,\gamma}^\mu$, does so via $\mu'\in\Pi(\lambda:\gamma)$ conjugate to $\mu$ via $\mathfrak{W}_0$.  But one can verify using elementary plane geometry that for any $\lambda,\gamma\in\Lambda_0$, $\tau_i\left(\overline{\Pi}(\lambda:\gamma)\right)\subset\overline{\Pi}(\lambda:\gamma)$ for each generating reflection $i=0,1,2$ of $\mathfrak{W}_0$.  This observation along with assumption (2) implies no reflections of length greater than or equal to one may contribute to the desired fusion coefficient and moreover $N_{\lambda,\gamma}^\mu=m_\lambda(\mu-\gamma)>0$.
\end{proof}




\section{Technical machinery}\label{machinery}

\par This section contains the main consequences of being a connected \'etale algebra in $\mathcal{C}(\mathfrak{g},k)$ when $\mathfrak{g}$ has rank 2, and basic inequalities involving quantum analogs which will allow a proof that such algebras are reasonably uncommon.  Lastly a sketch of the proof of Theorem \ref{thebigone} is illustrated by example in the case $\mathfrak{g}=\mathfrak{sl}_2$.

\subsection{\'Etale algebra restrictions in $\mathcal{C}(\mathfrak{g},k)$}\label{machinerya}

\par Let $A$ be a connected \'etale algebra in $\mathcal{C}:=\mathcal{C}(\mathfrak{g},k)$ where $\mathfrak{g}$ is $\mathfrak{sl}_3$, $\mathfrak{so}_5$, or $\mathfrak{g}_2$ and $(\ell,m)\subset A$ be a nontrivial summand of $A$ which is \emph{minimal} in the sense that $\ell+m$ is minimal in the case of $\mathfrak{sl}_3$ and $\mathfrak{so}_5$, and $\ell+(3/2)m$ is minimal in the case of $\mathfrak{g}_2$.  The reasons for this distinction will be explained in the proof of Lemma \ref{simplefree}.

\begin{note}
Our goal is not to reprove Theorem \ref{thebigone} in the rank 1 case so it will be satisfactory to point out the following lemmas can be restated for $\mathcal{C}(\mathfrak{sl}_2,k)$ where $(\ell)$ is the analogous minimal nontrivial summand of $A\in\mathcal{C}(\mathfrak{sl}_2,k)$.
\end{note}

\begin{lem}\label{simplefree}
If $(s,t)\in\Lambda_0$ and $2(s+t)<\ell+m$ in the case $\mathfrak{g}=\mathfrak{sl}_3,\mathfrak{so}_5$, or $2(s+(3/2)t)<\ell+(3/2)m$ in the case $\mathfrak{g}=\mathfrak{g}_2$, then $(s,t)\otimes A$ is a simple right $A$-module.
\end{lem}

\begin{proof} Label $\lambda:=(s,t)$.  Then we have by \cite[Lemma 2, Lemma 4]{Ostrik2003}
\begin{equation}
\Hom_{\mathcal{C}_A}(\lambda\otimes A,\lambda\otimes A)=\Hom_\mathcal{C}(\lambda,\lambda\otimes A)=\Hom_\mathcal{C}(\lambda\otimes\lambda^\ast,A).\label{thatone1}
\end{equation}
The highest weight in $\Pi(\lambda:\lambda^\ast)$ is $\gamma:=(s+t,s+t)$ when $\mathfrak{g}=\mathfrak{sl}_3$ and $\gamma:=(2s,2t)$ when $\mathfrak{g}=\mathfrak{so}_5,\mathfrak{g}_2$.  The respective assumptions on $(s,t)$ relative to $(\ell,m)$ in our hypotheses imply $\gamma\neq(\ell,m)$ and it remains to check no other weights $(s',t')\in\Pi(\lambda:\lambda^\ast)$ are equal to $(\ell,m)$ either.  To this end it will suffice to check $\gamma-\alpha\neq(\ell,m)$ for each simple root $\alpha$ since our claim follows inductively on dominance ordering.  If $\mathfrak{g}=\mathfrak{sl}_3$, $\gamma-\alpha_1=(s+t-2,s+t+1)$ which is not equal to $(\ell,m)$ since $s+t-2+s+t+1=2(s+t)-1<\ell+m$ and symmetrically for $\alpha_2$.  If $\mathfrak{g}=\mathfrak{so}_5$, $\gamma-\alpha_1=(2s-2,2t+2)$ which is not equal to $(\ell,m)$ since $2s-2+2t+2=2(s+t)<\ell+m$ and $\gamma-\alpha_2=(2s+1,2t-2)$ which is not equal to $(\ell,m)$ since $2s+1+2t-2=2(s+2)-1<\ell+m$.  Lastly if  $\mathfrak{g}=\mathfrak{g}_2$, $\gamma-\alpha_1=(2s-2,2t+1)$ which is not equal to $(\ell,m)$ since $2s-2+(3/2)(2t+1)=2(s+(3/2)t)-1/2<\ell+(3/2)m$, and $\gamma-\alpha_2=(2s+3,2t-2)$ which is not equal to $(\ell,m)$ since $2s+3+(3/2)(2t-2)=2(s+(3/2)t)<\ell+(3/2)m$.  Moreover the right-hand side of (\ref{thatone1}) is one-dimensional and $\lambda\otimes A$ is a simple object in $\mathcal{C}_A$.
\end{proof}

\begin{lem}\label{simplesummand}
If $M\in\mathcal{C}_A$, and $(s,t)\subset M$ satisfies the hypotheses of Lemma \ref{simplefree}, then $(s,t)\otimes A$ is a right $A$-submodule of $M$.
\end{lem}

\begin{proof}
As in the proof of Lemma \ref{simplefree} with $\lambda:=(s,t)$, compute
\begin{equation*}\Hom_{\mathcal{C}_A}(\lambda\otimes A,M)=\Hom_\mathcal{C}(\lambda,M).\end{equation*}
By assumption and Lemma \ref{simplefree}, $\lambda\otimes A$ is simple, hence the result is proven since the right-hand side is nontrivial.
\end{proof}
\begin{cor}\label{goal}
For all $(s,t)\in\Lambda_0$ and $\{(s_i,t_i)\}_{i\in I}$, collections of simple summands of $M:=(s,t)\otimes A$ satisying the assumptions of Lemma \ref{simplefree},
\begin{equation*}\sum_{i\in I}\dim'(s_i,t_i)\leq\dim'(s,t).\end{equation*}
\end{cor}
\begin{proof}
Apply Lemma \ref{simplesummand} to $M$ along with each element of $\{(s_i,t_i)\}_{i\in I}$. For each $(s_i,t_i)$ we then have $(s_i,t_i)\otimes A\subset(s,t)\otimes A$.  Taking dimensions of the containment provides the inequality, then $\dim(A)$ can be divided out (since $\mathcal{C}(\mathfrak{g},k)$ is pseudo-unitary) and denominators cleared.
\end{proof}


\subsection{Quantum inequalities}\label{sec:inequalities}

\par Exact computations are often intractable with quantum analogs so we now collect a set of results that will be used frequently in the sequel to verify when inequalities of the type in Corollary \ref{goal} are true or false.  An illustration of the well-known formula for $q^n-q^{-n}$ in terms of sines when $q$ is a root of unity can be found in \cite[Figure 3]{schopieray2017} (there is an analogous formula in terms of cosine for $q^n+q^{-n}$).  Set $\varepsilon(\mathfrak{g},k)$ to be the denominator of $\ln q$ (see Figure \ref{fig:q}) so the following can be stated in the necessary generality.  
\begin{lem}\label{quantumtriangle}
If $n,m\in\mathbb{Z}_{\geq1}$, then $[n+m]\leq[n]+m$.
\end{lem}
\begin{proof}
We will present a proof for even $m$, leaving the near-identical case of odd $m$ to the reader.  Carrying out the long division and simplifying yields
\begin{align*}
[n+m]-[n]=&\left(\dfrac{q^{n+m}-q^{-(n+m)}}{q-q^{-1}}\right)-\left(\dfrac{q^n-q^{-n}}{q-q^{-1}}\right) \\
	=&\left(q^{n+m-1}+q^{n+m-3}+\cdots+q^{-(n+m-3)}+q^{-(n+m-1)}\right) \\
	&-\left(q^{n-1}+q^{n-3}+\cdots+q^{-(n-3)}+q^{-(n-1)}\right) \\
=&\sum_{i=1}^{m/2}(q^{n-1+2i}+q^{-(n-1+2i)}) \\
=&2\sum_{i=1}^{m/2}\cos\left(\dfrac{(n-1+2i)\pi}{\varepsilon(\mathfrak{g},k)}\right) \\
\leq&m
\end{align*}
by the triangle inequality.
\end{proof}
\begin{cor}\label{quantumtrianglecor}
If $n\in\mathbb{Z}_{\geq1}$, then $[n]\leq n$.
\end{cor}
\begin{lem}\label{lowerquantumbound}
If $n\in\mathbb{Z}_{\geq1}$ and $n\leq\dfrac{1}{2}\varepsilon(\mathfrak{g},k)$, then $[n]\geq\dfrac{1}{2}n$.
\end{lem}
\begin{proof}
Note that
\begin{equation*}[n]=\dfrac{\sin\left(\dfrac{n\pi}{\varepsilon(\mathfrak{g},k)}\right)}{\sin\left(\dfrac{\pi}{\varepsilon(\mathfrak{g},k)}\right)}.\end{equation*}
We have $0\leq n\leq\varepsilon(\mathfrak{g},k)$ by assumption so we may use the inequalities $\sin(x)\geq x(1-x/\pi)$ (for $0\leq x\leq\pi$) and $1/\sin(x)\geq1/x$ (for $x>0$) to yield
\begin{align*}
[n]&\geq\left(\dfrac{\varepsilon(\mathfrak{g},k)}{\pi}\right)\left(\dfrac{n\pi}{\varepsilon(\mathfrak{g},k)}\right)\left(1-\dfrac{n}{\varepsilon(\mathfrak{g},k)}\right) \\
&=n\left(1-\dfrac{n}{\varepsilon(\mathfrak{g},k)}\right) \\
&\geq\dfrac{1}{2}n.
\end{align*}
\end{proof}


\subsection{Exceptional algebras}

\par In \cite{KiO}, connected \'etale algebras in $\mathcal{C}(\mathfrak{sl}_2,k)$ are organized into an ADE classification scheme paralleling the classification of simply-laced Dynkin diagrams.  The connected \'etale algebra of ``type A'' is the trivial one given by the unit object $\1\in\mathcal{C}(\mathfrak{sl}_2,k)$ (Example \ref{twotwotwo}).  Those connected \'etale algebras of ``type D'' arise at even levels in the following manner.  The fusion subcategory $\mathcal{C}(\mathfrak{sl}_2,2k)_\text{pt}\subset\mathcal{C}(\mathfrak{sl}_2,2k)$ generated by invertible objects is equivalent to $\Rep(\mathbb{Z}/2\mathbb{Z})$ and connected \'etale algebras in $\Rep(\mathbb{Z}/2\mathbb{Z})$ are in one-to-one correspondence with subgroups of $\mathbb{Z}/2\mathbb{Z}$ as the additional cohomological data from \cite[Theorem 3.1]{Ostrikdouble} is trivial for cyclic groups.  ``Type A'' algebras correspond to the trivial subgroup in the ``type D'' construction, so we will refer to both types as \emph{standard} in this exposition, and any algebra that does not arise from this construction as \emph{exceptional}.

\begin{example}\label{ex:sun}
Extending the notation from Section \ref{sec:a2}, simple objects of $\mathcal{C}(\mathfrak{sl}_n,nk)$ for $k\in\mathbb{Z}_{\geq1}$ are enumerated by positive integer $(n-1)$-tuples $(s_1,s_2,\ldots,s_{n-1})$ such that $s_1+s_2+\cdots+s_{n-1}\leq nk$.  The fusion subcategory $\mathcal{C}(\mathfrak{sl}_n,nk)_{\text{pt}}\simeq\Rep(\mathbb{Z}/n\mathbb{Z})$ has simple objects $(s_1,s_2,\ldots,s_{n-1})$ such that $s_i=nk$ and $s_j=0$ for all $j\neq i$, along with the trivial object.  Standard connected \'etale algebras in $\mathcal{C}(\mathfrak{sl}_n,nk)$ are again in one-to-one correspondence with subgroups of $\mathbb{Z}/n\mathbb{Z}$.  All exceptional connected \'etale algebras in $\mathcal{C}(\mathfrak{sl}_2,k)$ are succinctly listed in \cite[Table 1]{KiO}, while all exceptional connected \'etale algebras in $\mathcal{C}(\mathfrak{sl}_3,k)$ are listed using modular invariants \cite[Equations 2.7d,2.7e,2.7g]{gannon} at levels $k=5,9,21$.  The theory of conformal embeddings provides examples of exceptional connected \'etale algebras in $\mathcal{C}(\mathfrak{sl}_4,k)$ at levels $k=4,6,8$, which are described in detail in \cite{coq1}.  Although there is no explicit proof in the current literature that there are even finitely many exceptional connected \'etale algebras in $\mathcal{C}(\mathfrak{sl}_4,k)$, it is presumed based on computational evidence that the aforementioned list is exhaustive.
\end{example}

\begin{example}\label{geetoo}
There are no nontrivial standard connected \'etale algebras in $\mathcal{C}(\mathfrak{g}_2,k)$ since $\mathcal{C}(\mathfrak{g}_2,k)_\text{pt}\simeq\Vep$, but there are two standard connected \'etale algebras in $\mathcal{C}(\mathfrak{so}_5,2k)$ since $\mathcal{C}(\mathfrak{so}_5,2k)_\text{pt}\simeq\Rep(\mathbb{Z}/2\mathbb{Z})$ corresponding to $(0,0)$ and $(0,0)\oplus(k,0)$.  For odd levels $k$, $\theta(k,0)=-1$ and so by Lemma \ref{twistisone}, $(0,0)\oplus(k,0)$ does not have the structure of a connected \'etale algebra.  As in Example \ref{ex:sun}, the theory of conformal embeddings provides examples of exceptional connected \'etale algebras in $\mathcal{C}(\mathfrak{g}_2,k)$ at levels $k=3,4$ and $\mathcal{C}(\mathfrak{so}_5,k)$ at levels $k=2,3,7,12$, which are described in detail in \cite{coq2}.
\end{example}


\subsection{Main theorem and proof outline}\label{outline}

\begin{theo}\label{thebigone}
If $\mathfrak{g}$ is of rank 2, there exist finitely many levels $k\in\mathbb{Z}_{\geq1}$ such that $\mathcal{C}(\mathfrak{g},k)$ contains an exceptional connected \'etale algebra.
\end{theo}

\par The proof of this result is contained in Sections \ref{a2}-\ref{g2} but illustrated below in the following example for $\mathfrak{sl}_2$.  A summary of the explicit bounds obtained can be found in Section \ref{conclusion}.

\begin{example}[$\mathcal{C}(\mathfrak{sl}_2,k)$]

\par If $A$ is an exceptional connected \'etale algebra in $\mathcal{C}(\mathfrak{sl}_2,k)$ with minimal nontrivial summand $(\ell)$, Lemma \ref{twistisone} applied to the twist formula in Section \ref{sec:a1} implies $(\ell)$ is in the root lattice, i.e. $\ell$ is even, say $\ell=2m$ for some $m\in\mathbb{Z}_{\geq1}$ and $2m<k$.  Explicit fusion rules for $\mathcal{C}(\mathfrak{sl}_2,k)$ are well-known \cite[Section 2.8]{DMNO}, and we see that $(m+3)\otimes(2m)$ contains summands $(m-1)$ and $(m-3)$ provided $3\leq m<k$.  Moreover Corollary \ref{goal} then implies
\begin{align}
&&[m]+[m-2]&<[m+4]\label{sl2ineq1}\\
&&&<[m]+4 \label{sl2ineq2}\\
\Rightarrow&&\dfrac{1}{2}(m-2)&<4\label{sl2ineq}
\end{align}
where (\ref{sl2ineq2}) results from applying Corollary \ref{quantumtrianglecor} to the right-hand side of (\ref{sl2ineq1}) and (\ref{sl2ineq}) results from applying Lemma \ref{lowerquantumbound} to the left-hand side of (\ref{sl2ineq2}) which is justified because $2m<k$ implies $m-2\leq(1/2)(k+2)$.  The inequality in (\ref{sl2ineq}) is false for $m>9$.  Moreover $\theta(\ell)=1$ by Lemma \ref{twistisone} and so $m(m+1)-2\geq k$ by analyzing the argument of the twist formula in Section \ref{sec:a1} which implies $k\leq88$ if $m\leq9$.
\end{example}
\begin{note}
It is possible to show from the definition of $[m]$ that the inequality in (\ref{sl2ineq1}) is false in a more restricted setting: $m>5$, which then implies $k\leq28$.  But there exists an exceptional connected \'etale algebra in $\mathcal{C}(\mathfrak{sl}_2,28)$ corresponding to the object $(0)\oplus(10)\oplus(18)\oplus(28)$ (type $E_8$ in the ADE classification \cite[Section 6]{KiO}) and so this bound is tight.  Even for Lie algebras of rank 2, computing precisely when such an inequality is true becomes unrealistically complex.  For the purposes of Theorem \ref{thebigone} \emph{any} bound will suffice.
\end{note}




\section{Proof of Theorem \ref{thebigone}: $\mathcal{C}(\mathfrak{sl}_3,k)$}\label{a2}

\par Let $A$ be an exceptional connected \'etale algebra in $\mathcal{C}(\mathfrak{sl}_3,k)$ with minimal nontrivial summand $(\ell,m)$ (i.e. $\ell+m$ is minimal).  Using duality ($(\ell,m)^\ast=(m,\ell)$, \cite[Corollary 3.2.1]{schopieray2017}) and rotation of $\Lambda_0$ by 120 degrees (tensoring with $(0,k)$), every $(\ell,m)\in\Lambda_0$ is conjugate to one $(\ell',m')$ such that $m'\leq\ell'\leq k/2$.  In what follows, the summands of $(\ell,m)^\ast\otimes(\ell,m)=(m,\ell)\otimes(\ell,m)$ will be computed and these summands are invariant under duality and rotation.  We will show $\ell'+m'$ is bounded for such a conjugate.  To do so we claim if $m\leq\ell\leq k/2$, then
\begin{equation}
\bigoplus_{i=0}^{\lfloor x/2\rfloor}(i,i)\subset(m,\ell)\otimes(\ell,m).\label{sl3sum}\end{equation}
The set $\overline{\Pi}(m,\ell:\ell,m)$, illustrated by example in Figure \ref{fig:su3} (refer to Section \ref{sec:geometry} for descriptions of the notation and visualization used), is a hexagon (triangle in the degenerate case $m=0$) with vertex $(0,0)$ and circumcenter $(\ell,m)$.  In particular $(i,i)\in \Pi(m,\ell:\ell,m)$ for $0\leq i\leq\lfloor x/2\rfloor$ (black nodes in Figure \ref{fig:su3}).  The angles formed between $\overline{\Pi}(m,\ell:\ell,m)$ and $T_1,T_2$ are 30 degrees when they exist.  Therefore, when folded over $T_1,T_2$, the edges of $\overline{\Pi}(m,\ell:\ell,m)$ containing $(0,0)$ are parallel to the line formed by the weights $(i,i)$, $0\leq i\leq\lfloor x/2\rfloor$, implying $\tau_j(\mu)\neq(i,i)$ for any $i\geq0$ and $j=1,2$.  Furthermore $m\leq\ell\leq k/2$ ensures there is no contribution from $\tau_0$ to $N_{(m,\ell),(\ell,m)}^{(i,i)}$ for any of the desired summands.  Lemma \ref{fusion} then implies containment (\ref{sl3sum}).
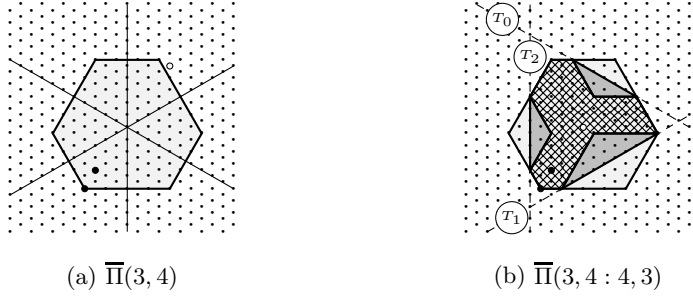
\begin{figure}[H]
\centering
\begin{subfigure}{.5\textwidth}
  \centering
\begin{tikzpicture}[scale=0.2]
\draw[fill=black!5,thick] (2*1.4142,5*0.8165) -- (-1*1.414,5*0.8165) -- (-3*1.414,-1*0.8165) -- (-1.5*1.414,-5.5*0.8165) -- (2.5*1.414,-5.5*0.8165) -- (4*1.414,-1*0.8165) -- cycle;
\draw[thin] (-5*1.4142,-6*0.8165) -- (5.5*1.4142,4.5*0.8165);
\draw[thin] (-5*1.4142,5*0.8165) -- (5.5*1.4142,-5.5*0.8165);
\draw[thin] (0.5*1.4142,9.5*0.8165) -- (0.5*1.4142,-9*0.8165);
\foreach \x in {-5,-4,...,5} {
	\foreach \y in {-9,-8,...,9} {
        \node at (\x*1.4142,\y*0.8165) {$\cdot$};
    		}
		};
\foreach \x in {-5,-4,...,5} {
	\foreach \y in {-9,-8,...,9} {
        \node at (\x*1.4142+0.7071,\y*0.8165+0.40825) {$\cdot$};
    		}
		};
\draw[fill=black] (-1.5*1.414,-5.5*0.8165) circle (0.2);
\draw[fill=black] (-1*1.414,-4*0.8165) circle (0.2);
\draw[fill=white] (2.5*1.4142,4.5*0.8165) circle (0.2);
\end{tikzpicture}
  \caption{$\overline{\Pi}(3,4)$}
  \label{fig:sl3123}
\end{subfigure}%
\begin{subfigure}{.5\textwidth}
  \centering
\begin{tikzpicture}[scale=0.2]
\draw[fill=black!5,thick] (2*1.4142,5*0.8165) -- (-1*1.414,5*0.8165) -- (-3*1.414,-1*0.8165) -- (-1.5*1.414,-5.5*0.8165) -- (2.5*1.414,-5.5*0.8165) -- (4*1.414,-1*0.8165) -- cycle;
\draw[fill=black!25,thick] (-2*1.414,-4*0.8165) -- (-2*1.414,2*0.8165) -- (-1*1.414,-1*0.8165) -- cycle;
\draw[fill=black!25,thick] (0*1.414,5*0.8165) -- (3*1.414,2*0.8165) -- (1*1.414,2*0.8165) -- cycle;
\draw[fill=black!25,thick] (-0.5*1.414,-5.5*0.8165) -- (1*1.414,-1*0.8165) -- (4*1.414,-1*0.8165) -- cycle;
\draw[pattern=crosshatch,thick] (-2*1.414,-4*0.8165) -- (-1*1.414,-1*0.8165) -- (-2*1.414,2*0.8165) -- (-1*1.414,5*0.8165)  -- (0*1.414,5*0.8165) -- (1*1.414,2*0.8165) -- (3*1.414,2*0.8165) -- (4*1.414,-1*0.8165) -- (1*1.414,-1*0.8165) -- (-0.5*1.414,-5.5*0.8165) -- (-1.5*1.414,-5.5*0.8165)  --  cycle;
\foreach \x in {-5,-4,...,5} {
	\foreach \y in {-9,-8,...,9} {
        \node at (\x*1.4142,\y*0.8165) {$\cdot$};
    		}
		};
\foreach \x in {-5,-4,...,5} {
	\foreach \y in {-9,-8,...,9} {
        \node at (\x*1.4142+0.707,\y*0.8165+0.40825) {$\cdot$};
    		}
		};
\draw[fill=white] (0.5*1.414,-0.5*0.8165) circle (0.2);
\draw[fill=black] (-1.5*1.414,-5.5*0.8165) circle (0.2);
\draw[fill=black] (-1*1.414,-4*0.8165) circle (0.2);
\draw[dashed] (-4*1.4142,-9*0.8165) -- node[circle,draw=black,pos = 0.12,inner sep=1pt,fill=white,solid,thin] {\tiny$T_1$} (5.5*1.4142,0.5*0.8165);
\draw[dashed] (-2*1.4142,9.5*0.8165) -- node[circle,draw=black,pos= 0.23,inner sep=1pt,fill=white,solid,thin] {\tiny$T_2$} (-2*1.4142,-9*0.8165);
\draw[dashed] (-4.5*1.4142,9.5*0.8165) -- node[circle,draw=black,pos = 0.12,inner sep=1pt,fill=white,solid,thin] {\tiny$T_0$} (5.5*1.4142,-0.5*0.8165);
\end{tikzpicture}
  \caption{$\overline{\Pi}(3,4:4,3)$}
  \label{fig:sl3321}
\end{subfigure}
\caption{$(3,4)\otimes(4,3)\in\mathcal{C}(\mathfrak{sl}_3,12)$}
\label{fig:su3}
\end{figure}
By Corollary \ref{goal}, containment (\ref{sl3sum}) implies
\begin{equation}
\sum_{i=0}^{\lfloor x/2\rfloor}[i+1]^2[2(i+1)]\leq[\ell+1][m+1][\ell+m+2],\label{sl3firsteq}
\end{equation}
while applying Corollary \ref{quantumtrianglecor} to the right-hand side of (\ref{sl3firsteq}) and Lemma \ref{lowerquantumbound} to the left-hand side of (\ref{sl3firsteq}) (which is applicable since $i\leq x/2$ implies $2(i+1)\leq x+2<k+3$) yields
\begin{equation}
\sum_{i=0}^{\lfloor x/2\rfloor}\left(\dfrac{1}{4}\right)(i+1)^2\left(\dfrac{1}{2}\right)2(i+1)\leq(\ell+1)(m+1)(\ell+m+2). \label{sl3sec}
\end{equation}
Furthermore we re-index the left-hand side of (\ref{sl3sec}), and bound each of the factors on the right-hand side of (\ref{sl3sec}) in terms of $x$ to produce
\begin{equation}
\frac{1}{4}\sum_{i=1}^{\lfloor x/2\rfloor+1}i^3\leq(2x+2)(x+2)(2x+4).\label{faul1}
\end{equation}
Now to eliminate the sum we proceed by parity: if $x$ is even $\lfloor x/2\rfloor+1=x/2+1$ and if $x$ is odd $\lfloor x/2\rfloor+1=x/2+1/2$.  Then using Faulhaber's formula (refer to the introduction of \cite{knuth} for a brief history and statement of this formula) on the left-hand side of (\ref{faul1}) implies the inequalities
\begin{align*}
(x\text{ even})&&\frac{1}{256}(x+2)^2(x+4)^2&\leq(2x+2)(x+2)(2x+4),\text{ and}\\
(x\text{ odd})&&\frac{1}{256}(x+1)^2(x+3)^2&\leq(2x+2)(x+2)(2x+4).
\end{align*}
The first inequality is true for even $x$ such that $x<1017$ while the second is true for odd $x$ such that $x<1021$.

\par Lemma \ref{twistisone} implies $\theta(\ell,m)=1$ for our original minimal nontrivial summand of $A$.  One consequence is that $(\ell,m)$ is contained in the root lattice inside $\Lambda_0$ (i.e. $\ell\equiv m\pmod{3}$).  Another consequence is that $\theta(\ell',m')$, the twist of its conjugate, is a third root of unity.  To see this note that $\theta(0,k)$ is a third root of unity depending on the level $k$ modulo 3 and $(\ell,m)$ is in the centralizer of the pointed subcategory generated by the simple object $(0,k)$ (refer to the proof of \cite[Proposition 3.4.1]{schopieray2017}).  Our claim then follows from the ribbon axioms $\theta((0,k)\otimes(\ell,m))=\theta(\ell,m)\theta(0,k)$, and $\theta(\ell,m)=\theta(m,\ell)^{-1}$ \cite[Definition 8.10.1]{tcat}.  

\par Furthermore, $\theta(\ell',m')$ being a third root of unity forces $(\ell'+3\ell'+\ell'm'+3m'+m'^2)/(k+3)\in\mathbb{Z}$ and moreover $(\ell'^2+3\ell'+\ell'm'+3m'+m'^2)-3\geq k$.  The left-hand side of this inequality is maximized (as a real symmetric function of $\ell',m'\geq0$) when $\ell'=m'$, which by the above argument can be no larger than $x\leq1019$.  Hence we have $k\leq3121194$.  In summary any exceptional connected \'etale algebra in $\mathcal{C}(\mathfrak{sl}_3,k)$ must have a minimal summand which is conjugate to $(\ell',m')$ such that $\ell'+m'\leq2038$ and must occur at a level $k\leq3121194$, proving Theorem \ref{thebigone} for $\mathcal{C}(\mathfrak{sl}_3,k)$.




\section{Proof of Theorem \ref{thebigone}: $\mathcal{C}(\mathfrak{so}_5,k)$}\label{b2}

\par Let $A$ be a connected \'etale algebra in $\mathcal{C}(\mathfrak{so}_5,k)$ with minimal nontrivial summand $(\ell,m)$ (i.e. $\ell+m$ is minimal) and let $x:=\lceil(1/2)(\ell+m)\rceil-1$, the greatest integer strictly less than the average of $\ell$ and $m$.  The quantity $x$ is crucial in the remainder of Section \ref{b2} as summands $(s,t)$ such that $s+t\leq x$ are precisely those which will satisfy the hypotheses of Lemma \ref{simplefree}.  We aim to provide an explicit bound on $x$ to subsequently produce a bound on the level $k$ for which such a connected \'etale algebra can exist.

\par Lemma \ref{twistisone} implies that $(\ell,m)$ lies in the root lattice (i.e. $m$ is even).  Our proof will be split into four cases (three of the four cases have an argument based on the parity of $\ell$), illustrated by example in Figure \ref{fig:cases}, based on the relative size of $m$ versus $x$: $m=0$ and $\ell<k-1$, $0\leq m-2\leq x$, $0\neq\ell\leq x<m-2$, and $\ell=0$ with $m<k$.  The case $(\ell,m)=(k,0)$ corresponds to either the standard connected \'etale algebra $(0,0)\oplus(k,0)$ (if $k$ is even; see Example \ref{geetoo}) or $A$ has a nontrivial minimal summand covered by another case.  In the case $(\ell,m)=(k-1,0)$, $\theta(k-1,0)=1$ if and only if $(k+2)(k-1)/(2(k+3))$ is an integer.  It can be easily verified that for $k\in\mathbb{Z}_{\geq1}$, $(k+2)(k-1)/(2(k+3))$ is an integer if and only if $k=1$.  Similarly $\theta(0,k)=1$ if and only if $k(k+4)/(k+3)$ is an integer which is likewise only the case when this integer is zero.  Moreover all possible $(\ell,m)$ will be discussed through these four cases.

\begin{figure}[H]
\centering
\begin{tikzpicture}[scale=0.2]
\draw[fill=gray!10,rounded corners=0.66mm] (-4.25*1.414,6.25*1.414) -- (-3.75*1.414,6.25*1.414) -- (-3.75*1.414,4.75*1.414)  -- (-4.25*1.414,4.75*1.414) -- cycle;
\draw[fill=gray!10,rounded corners=0.66mm] (-3.25*1.414,5.66*1.414) -- (-0.75*1.414,3.166*1.414) -- (-0.75*1.414,1.333*1.414)  -- (-3.25*1.414,3.833*1.414) -- cycle;
\draw[fill=gray!10,rounded corners=0.66mm] (-0.25*1.414,2.66*1.414) -- (1.25*1.414,1.166*1.414) -- (1.25*1.414,-0.666*1.414)  -- (-0.25*1.414,0.833*1.414) -- cycle;
\draw[fill=white!10,rounded corners=0.66mm] (1.75*1.414,0.25*1.414) -- (2.25*1.414,0.25*1.414) -- (2.25*1.414,-0.25*1.414)  -- (1.75*1.414,-0.25*1.414) -- cycle;
\foreach \x in {-6,-5,...,4} {
	\foreach \y in {-8,-7,...,10} {
        \node at (\x*1.4142,\y*1.4142) {$\cdot$};
    		}
		};
\foreach \x in {-6,-5,...,4} {
	\foreach \y in {-8,-7,...,9} {
        \node at (\x*1.4142+0.707,\y*1.4142+0.707) {$\cdot$};
    		}
		};
\draw[fill=black] (-4*1.414,6*1.414) circle (0.2);
\draw[fill=black] (-4*1.414,5*1.414) circle (0.2);
\draw[fill=black] (-3*1.414,5*1.414) circle (0.2);
\draw[fill=black] (-3*1.414,4*1.414) circle (0.2);
\draw[fill=black] (-2*1.414,4*1.414) circle (0.2);
\draw[fill=black] (-2*1.414,3*1.414) circle (0.2);
\draw[fill=black] (-1*1.414,3*1.414) circle (0.2);
\draw[fill=black] (-1*1.414,2*1.414) circle (0.2);
\draw[fill=black] (0*1.414,2*1.414) circle (0.2);
\draw[fill=black] (0*1.414,1*1.414) circle (0.2);
\draw[fill=black] (1*1.414,1*1.414) circle (0.2);
\draw[fill=black] (1*1.414,0*1.414) circle (0.2);
\draw[fill=black] (2*1.414,0*1.414) circle (0.2);

\draw[dashed,thick] (-4.5*1.414,-8*1.414) -- node[circle,draw=black,pos = 0.5,inner sep=1pt,fill=white,solid,thin] {\tiny$T_1$}  (-4.5*1.414,10*1.414);
\draw[dashed,thick] (-5*1.414,-8*1.414) -- node[circle,draw=black,pos = 0.5,inner sep=1pt,fill=white,solid,thin] {\tiny$T_2$}  (4.5*1.414,1.5*1.414);
\draw[dashed,thick] (4.5*1.414,0.5*1.414) -- node[circle,draw=black,pos = 0.5,inner sep=1pt,fill=white,solid,thin] {\tiny$T_0$}  (-5*1.414,10*1.414);
\end{tikzpicture}
\caption{Possible $(\ell,m)$ when $k=14$ and $x=5$}
\label{fig:cases}
\end{figure}
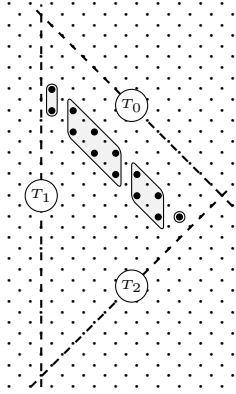

\vspace{-8 mm}


\subsection{The case $m=0$ and $0<\ell<k-1$}\label{sog1}

\par Set $\lambda:=\ell-x+2$ so we have $\lambda=x+4$ if $\ell$ is even and $\lambda=x+3$ if $\ell$ is odd.  We claim that if $5\leq\ell<k-1$, then
\begin{equation}
(\ell-\lambda,0)\oplus(\ell-\lambda,2)\subset(\lambda,0)\otimes(\ell,0).\label{sl3caseoneineq}
\end{equation}
The set $\overline{\Pi}(\lambda,0)$, illustrated by example in Figure \ref{fig:test}, (refer to Section \ref{sec:geometry} for descriptions of the notation and visualization used) is a square with vertex $(-\lambda,0)$ and its three conjugates under the Weyl group.  In particular {$\Pi(\lambda,0:\ell,0)$} contains $(\ell-\lambda,0)$ and $(\ell-\lambda,2)$ provided $\ell\geq5$.  The reflection $\tau_1$ cannot contribute to $N_{(\lambda,0),(\ell,0)}^{(\ell-\lambda,2)}$ or $N_{(\lambda,0),(\ell,0)}^{(\ell-\lambda,0)}$ as $(\lambda,0)$ does not lie on $T_1$, nor does $\tau_0$ contribute by the assumption $\ell<k-1$.  There can be no contribution from $\tau_2$ as $\overline{\Pi}(\lambda,0:\ell,0)$ does not intersect $T_2$, thus Lemma \ref{fusion} implies (\ref{sl3caseoneineq}).
\begin{figure}[H]
\centering
\begin{subfigure}{.5\textwidth}
  \centering
\begin{tikzpicture}[scale=0.18]
\draw[fill=black!5,thick] (0*1.414,6*1.414) -- (6*1.414,0*1.414) -- (0*1.414,-6*1.414) -- (-6*1.414,0*1.414) -- cycle;
\foreach \x in {-6,-5,...,6} {
	\foreach \y in {-9,-8,...,8} {
        \node at (\x*1.4142,\y*1.4142) {$\cdot$};
    		}
		};
\foreach \x in {-7,-6,...,6} {
	\foreach \y in {-9,-8,...,8} {
        \node at (\x*1.4142+0.707,\y*1.4142+0.707) {$\cdot$};
    		}
		};

\draw (-6.5*1.414,0*1.414) -- (6.5*1.414,0*1.414);
\draw (0*1.414,-9*1.414) -- (0*1.414,8.5*1.414);
\draw (-6.5*1.414,-6.5*1.414) -- (6.5*1.414,6.5*1.414);
\draw (-6.5*1.414,6.5*1.414) -- (6.5*1.414,-6.5*1.414);
\draw[fill=white] (0*1.414,7*1.414) circle (0.3);
\draw[fill=black] (0*1.414,-6*1.414) circle (0.3);
\draw[fill=black] (1*1.414,-5*1.414) circle (0.3);
\end{tikzpicture}
  \caption{$\overline{\Pi}(6,0)$}
  \label{fig:sub}
\end{subfigure}%
\begin{subfigure}{.5\textwidth}
  \centering
\begin{tikzpicture}[scale=0.18]
\draw[fill=black!5,thick] (0*1.414,6*1.414) -- (6*1.414,0*1.414) -- (0*1.414,-6*1.414) -- (-6*1.414,0*1.414) -- cycle;
\draw[pattern=crosshatch,thick] (-0.5*1.414,3.5*1.414) -- (4*1.414,-2*1.414) -- (0*1.414,-6*1.414) -- (-0.5*1.414,-5.5*1.414) -- cycle;
\draw[fill=black!25,thick] (-0.5*1.414,3.5*1.414) -- (-0.5*1.414,1.5*1.414) -- (3.5*1.414,-2.5*1.414)-- (4.5*1.414,-1.5*1.414) -- cycle;
\draw[fill=black!25,thick] (-0.5*1.414,1.5*1.414) -- (3*1.414,-2*1.414) -- (-0.5*1.414,-5.5*1.414) -- cycle;
\foreach \x in {-6,-5,...,6} {
	\foreach \y in {-9,-8,...,8} {
        \node at (\x*1.4142,\y*1.4142) {$\cdot$};
    		}
		};
\foreach \x in {-7,-6,...,6} {
	\foreach \y in {-9,-8,...,8} {
        \node at (\x*1.4142+0.707,\y*1.4142+0.707) {$\cdot$};
    		}
		};
\draw[dashed,thick] (-0.5*1.414,-9*1.414) -- node[circle,draw=black,pos = 0.95,inner sep=1pt,fill=white,solid,thin] {\tiny$T_1$}  (-0.5*1.414,8.5*1.414);
\draw[dashed,thick] (-1*1.414,-9*1.414) -- node[circle,draw=black,pos = 0.6,inner sep=1pt,fill=white,solid,thin] {\tiny$T_2$}  (6.5*1.414,-1.5*1.414);
\draw[dashed,thick] (6.5*1.414,-3.5*1.414) -- node[circle,draw=black,pos = 0.8,inner sep=1pt,fill=white,solid,thin] {\tiny$T_0$}  (-5.5*1.414,8.5*1.414);
\draw[fill=white] (0*1.414,0*1.414) circle (0.3);
\draw[fill=black] (0*1.414,-6*1.414) circle (0.3);
\draw[fill=black] (1*1.414,-5*1.414) circle (0.3);
\end{tikzpicture}
  \caption{$\overline{\Pi}(6,0:7,0)$}
  \label{fig:sub1a}
\end{subfigure}
\caption{$(6,0)\otimes(7,0)\in\mathcal{C}(\mathfrak{so}_5,9)$}
\label{fig:test}
\end{figure}
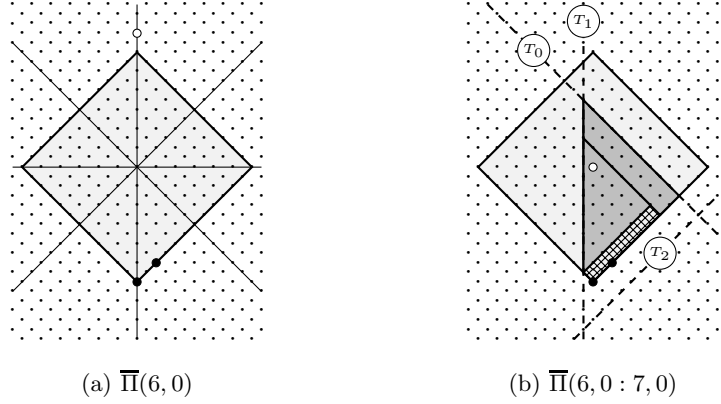

\par If $\ell$ is even, Corollary \ref{goal} applied to (\ref{sl3caseoneineq}) gives
\begin{align}
\hspace{-2mm}\dim'(x-2,0)+\dim'(x-2,2)&\leq[2x+10][2x+11][2x+12]\label{eq:8} \\
&\leq([2x-2]+12)([2x-1]+12)([2x]+12)\label{eq:100}
\end{align}
by applying Lemma \ref{quantumtriangle} to the right-hand side of (\ref{eq:8}).  Then expanding the product in (\ref{eq:100}) and subtracting the leading term (equal to $\dim'(x-2,0)$) from both sides yields
\begin{equation}
[3][2x-2][2x+4][2x+1]\leq24(6x^2+30x+55)\label{eq:9}
\end{equation}
using Corollary \ref{quantumtrianglecor} on the right-hand side to eliminate the quantum analogs.  Moreover, applying Lemma \ref{lowerquantumbound} to the left-hand side of (\ref{eq:9}) (which is justified since $x=(1/2)\ell-1$ implies $2(2x+4)\leq2(k+3)$) leaves the inequalities
\begin{align}
(\ell\text{ even})&&\dfrac{3}{4}(x-1)(2x+1)(x+2)&\leq24(6x^2+30x+55),\text{ and}\label{eq:10} \\
(\ell\text{ odd})&&\dfrac{3}{4}(x-1)(2x+1)(x+2)&\leq120(x^2+4x+6)\label{eq:11}
\end{align}
repeating the same process for $\ell$ odd. Inequality (\ref{eq:10}) is true for even $\ell$ with $x\leq98$ and inequality (\ref{eq:11}) is true for odd $\ell$ with $x\leq81$.  The former is a weaker bound on $\ell=2x+2\leq198$, which using $\theta(\ell,0)=1$ by Lemma \ref{twistisone} implies $(2\ell^2+6\ell)/(4(k+3))\in\mathbb{Z}$ and thus $k\leq(2(198)^2+6(198))/4-3=19896$.



\subsection{The case $2\leq m\leq x+2$}\label{soo2}

Set $\lambda:=\ell+m-x$ so that $\lambda=x+1$ when $\ell$ is odd and $\lambda=x+2$ if $\ell$ is even.  We claim that for $2\leq m\leq x+2$,
\begin{equation}
(x,0)\oplus(x-2,2)\subset(\lambda,0)\otimes(\ell,m).\label{so5eq4}
\end{equation}
The set $\overline{\Pi}(\lambda,0)$, illustrated by example in Figure \ref{fig:test3}, is a square with vertex $(-\lambda,0)$ and its three conjugates under the Weyl group.  From the fact $m\geq2$ is even, the set $\Pi(\lambda,0:\ell,m)$ contains $(x,0)$ and $(x-2,2)$.  The square $\overline{\Pi}(\lambda,0:\ell,m)$ intersects $T_1$ at 45 degree angles, thus $(x,0)$ and $(x-2,2)$ lying on this intersecting edge implies there is no contribution to the desired fusion coefficients from $\tau_1$.  Reflection $\tau_0$ could only contribute if $(\ell,m)$ lies on $T_0$, and the assumption $m\leq x+2$ ensures there is no contribution from $\tau_2$ as well.  Lemma \ref{fusion} then implies containment (\ref{so5eq4}).

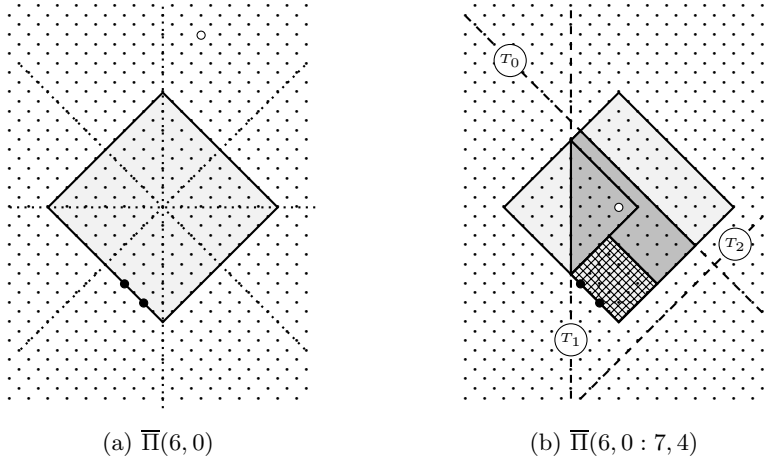
\begin{figure}[H]
\centering
\begin{subfigure}{.5\textwidth}
  \centering
\begin{tikzpicture}[scale=0.18]
\draw[fill=black!5,thick] (0*1.414,6*1.414) -- (6*1.414,0*1.414) -- (0*1.414,-6*1.414) -- (-6*1.414,0*1.414) -- cycle;
\foreach \x in {-8,-7,...,7} {
	\foreach \y in {-10,-9,...,10} {
        \node at (\x*1.4142,\y*1.4142) {$\cdot$};
    		}
		};
\foreach \x in {-8,-7,...,7} {
	\foreach \y in {-10,-9,...,10} {
        \node at (\x*1.4142+0.707,\y*1.4142+0.707) {$\cdot$};
    		}
		};

\draw[dotted,thick] (0:-8*1.414) -- (0:8*1.414);
\draw[dotted,thick] (90:-10.5*1.414) -- (90:10.5*1.414);
\draw[dotted,thick] (-45:-10.5*1.414) -- (-45:10.5*1.414);
\draw[dotted,thick] (45:-10.5*1.414) -- (45:10.5*1.414);
\draw[fill=white] (2*1.414,9*1.414) circle (0.3);
\draw[fill=black] (-1*1.414,-5*1.414) circle (0.3);
\draw[fill=black] (-2*1.414,-4*1.414) circle (0.3);
\end{tikzpicture}
  \caption{$\overline{\Pi}(6,0)$}
  \label{fig:sub1}
\end{subfigure}%
\begin{subfigure}{.5\textwidth}
  \centering
\begin{tikzpicture}[scale=0.18]
\draw[fill=black!5,thick] (0*1.414,6*1.414) -- (6*1.414,0*1.414) -- (0*1.414,-6*1.414) -- (-6*1.414,0*1.414) -- cycle;
\draw[fill=black!25,thick] (-2.5*1.414,-3.5*1.414) -- (1*1.414,0*1.414) -- (-2.5*1.414,3.5*1.414) -- cycle;
\draw[fill=black!25,thick] (-2.5*1.414,3.5*1.414) -- (-2*1.414,4*1.414) --(4*1.414,-2*1.414) --(2*1.414,-4*1.414) --(-0.5*1.414,-1.5*1.414) --(1*1.414,0*1.414) --cycle;
\draw[pattern=crosshatch,thick] (-2.5*1.414,-3.5*1.414) -- (-0.5*1.414,-1.5*1.414) -- (2*1.414,-4*1.414) -- (0*1.414,-6*1.414) -- cycle;
\foreach \x in {-8,-7,...,7} {
	\foreach \y in {-10,-9,...,10} {
        \node at (\x*1.4142,\y*1.4142) {$\cdot$};
    		}
		};
\foreach \x in {-8,-7,...,7} {
	\foreach \y in {-10,-9,...,10} {
        \node at (\x*1.4142+0.707,\y*1.4142+0.707) {$\cdot$};
    		}
		};
\draw[dashed,thick] (-2.5*1.414,-10*1.414) -- node[circle,draw=black,pos = 0.15,inner sep=1pt,fill=white,solid,thin] {\tiny$T_1$}   (-2.5*1.414,10.5*1.414);
\draw[dashed,thick] (-2*1.414,-10*1.414) -- node[circle,draw=black,pos = 0.85,inner sep=1pt,fill=white,solid,thin] {\tiny$T_2$}   (7.5*1.414,-0.5*1.414);
\draw[dashed,thick] (7.5*1.414,-5.5*1.414) -- node[circle,draw=black,pos = 0.85,inner sep=1pt,fill=white,solid,thin] {\tiny$T_0$}   (-8*1.414,10*1.414);
\draw[fill=white] (0*1.414,0*1.414) circle (0.3);
\draw[fill=black] (-1*1.414,-5*1.414) circle (0.3);
\draw[fill=black] (-2*1.414,-4*1.414) circle (0.3);
\end{tikzpicture}
  \caption{$\overline{\Pi}(6,0:7,4)$}
  \label{fig:sub3a}
\end{subfigure}
\caption{$(6,0)\otimes(7,4)\in\mathcal{C}(\mathfrak{so}_5,12)$}
\label{fig:test3}
\end{figure}

\par If $\ell$ is odd, Corollary \ref{goal} applied to (\ref{so5eq4}) gives
\begin{align}
&&\dim'(x,0)+\dim'(x-2,2)&\leq[2(x+2)][2(x+3)][2x+5] \label{eq:4} \\
&&&\leq([2x+2]+2)([2x+3]+2)([2x+4]+2)\label{eq:4b}
\end{align}
using Lemma \ref{quantumtriangle} on the right-hand side of (\ref{eq:4}).  Expanding the product on the right-hand side of (\ref{eq:4b}) and subtracting the leading term (equal to $\dim'(x,0)$) yields
\begin{equation}
[3][2(x-1)][2x+1][2(x+2)]\leq24(x+2)^2\label{so5eq10}
\end{equation}
using Corollary \ref{quantumtrianglecor} on the right-hand side.  Applying Lemma \ref{lowerquantumbound} to the left-hand side of (\ref{so5eq10}) is justified since $2(2x+4)=2(\ell+m+3)\leq2(k+3)$ and thus
\begin{align}
(\ell\text{ odd})&&\dfrac{3}{4}(x-1)(2x+1)(x+2)&\leq24(x+2)^2,\text{ and}\label{eq:so51} \\
(\ell\text{ even})&&\dfrac{3}{4}(x-1)(2x+1)(x+2)&\leq24(2x^2+10x+13)\label{eq:so52}
\end{align}
repeating the above process for $\ell$ even.  The inequality in (\ref{eq:so51}) is true for odd $\ell$ with $x\leq18$ while the inequality in (\ref{eq:so52}) is true for even $\ell$ with $x\leq35$.  Moreover $2\leq m\leq37$, $\ell+m\leq72$, and therefore $k\leq2625$ from Lemma \ref{twistisone} by maximizing $(2\ell^2+2\ell m+6\ell+m^2+4m)/4-3$ subject to these constraints as in the conclusion of Section \ref{sog1}.


\subsection{The case $\ell=0$ and $m<k$}

\par We claim for $m\geq4$,
\begin{equation}
\bigoplus_{i=0}^{x-1}(i,0)\subset(0,m)\otimes(0,m).\label{so5eq2}
\end{equation}
The set $\overline{\Pi}(0,m)$, illustrated by example in Figure \ref{fig:test2}, is a square with vertex $(0,-m)$ and its three conjugates under the Weyl group.  In particular $\Pi(0,m:0,m)$ contains $(i,0)$ for $0\leq i\leq x-1$.  The angles formed between $T_0,T_2$ and $\overline{\Pi}(0,m:0,m)$ are 45 degrees, ensuring there is no contribution to the desired fusion coefficients from $\tau_0,\tau_2$; $\overline{\Pi}(0,m:0,m)$ does not intersect $T_1$ so there is no contribution from $\tau_1$ either.  Lemma \ref{fusion} then implies containment (\ref{so5eq2}).

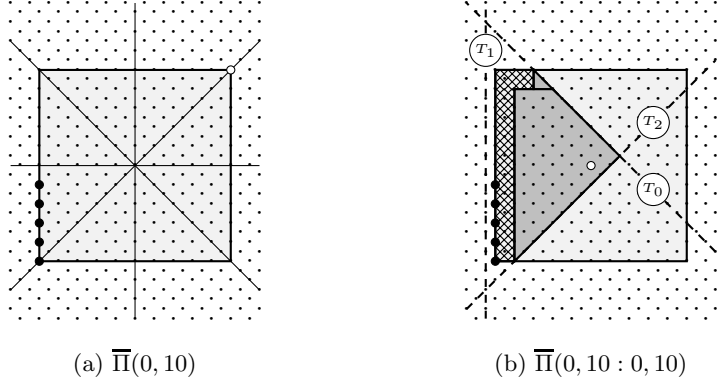
\begin{figure}[H]
\centering
\begin{subfigure}{.5\textwidth}
  \centering
\begin{tikzpicture}[scale=0.18]
\draw[fill=black!5,thick] (5*1.414,5*1.414) -- (5*1.414,-5*1.414) -- (-5*1.414,-5*1.414) -- (-5*1.414,5*1.414) -- cycle;
\foreach \x in {-6,-5,...,6} {
	\foreach \y in {-8,-7,...,8} {
        \node at (\x*1.4142,\y*1.4142) {$\cdot$};
    		}
		};
\foreach \x in {-7,-6,...,6} {
	\foreach \y in {-8,-7,...,8} {
        \node at (\x*1.4142+0.707,\y*1.4142+0.707) {$\cdot$};
    		}
		};

\draw[thin] (0*1.414,-8*1.414) -- (0*1.414,8.5*1.414);
\draw[thin] (-6.5*1.414,6.5*1.414) -- (6.5*1.414,-6.5*1.414);
\draw[thin] (-6.5*1.414,-6.5*1.414) -- (6.5*1.414,6.5*1.414);
\draw[thin] (-6.5*1.414,0*1.414) -- (6.5*1.414,0*1.414);
\draw[fill=white] (5*1.414,5*1.414) circle (0.3);
\draw[fill=black] (-5*1.414,-5*1.414) circle (0.3);
\draw[fill=black] (-5*1.414,-4*1.414) circle (0.3);
\draw[fill=black] (-5*1.414,-3*1.414) circle (0.3);
\draw[fill=black] (-5*1.414,-2*1.414) circle (0.3);
\draw[fill=black] (-5*1.414,-1*1.414) circle (0.3);
\end{tikzpicture}
  \caption{$\overline{\Pi}(0,10)$}
  \label{fig:sub1}
\end{subfigure}%
\begin{subfigure}{.5\textwidth}
  \centering
\begin{tikzpicture}[scale=0.18]
\draw[fill=black!5,thick] (5*1.414,5*1.414) -- (5*1.414,-5*1.414) -- (-5*1.414,-5*1.414) -- (-5*1.414,5*1.414) -- cycle;
\draw[fill=black!25,thick] (-4*1.414,4*1.414) -- (-2*1.414,4*1.414)-- (1.5*1.414,0.5*1.414) -- (-2*1.414,-3*1.414) -- (-4*1.414,-5*1.414)  -- cycle;
\draw[fill=black!25,thick] (-3*1.414,4*1.414) -- (-3*1.414,5*1.414) -- (-2*1.414,4*1.414) -- cycle;
\draw[pattern=crosshatch,thick] (-5*1.414,-5*1.414) -- (-4*1.414,-5*1.414) -- (-4*1.414,4*1.414) -- (-3*1.414,4*1.414) -- (-3*1.414,5*1.414) -- (-5*1.414,5*1.414) -- cycle;
\foreach \x in {-6,-5,...,6} {
	\foreach \y in {-8,-7,...,8} {
        \node at (\x*1.4142,\y*1.4142) {$\cdot$};
    		}
		};
\foreach \x in {-7,-6,...,6} {
	\foreach \y in {-8,-7,...,8} {
        \node at (\x*1.4142+0.707,\y*1.4142+0.707) {$\cdot$};
    		}
		};
\draw[dashed,thick] (-5.5*1.414,-8*1.414) -- node[circle,draw=black,pos = 0.85,inner sep=1pt,fill=white,solid,thin] {\tiny$T_1$}   (-5.5*1.414,8.5*1.414);
\draw[dashed,thick] (-6.5*1.414,-7.5*1.414) -- node[circle,draw=black,pos = 0.75,inner sep=1pt,fill=white,solid,thin] {\tiny$T_2$}   (6.5*1.414,5.5*1.414);
\draw[dashed,thick] (6.5*1.414,-4.5*1.414) -- node[circle,draw=black,pos = 0.25,inner sep=1pt,fill=white,solid,thin] {\tiny$T_0$}   (-6.5*1.414,8.5*1.414);
\draw[fill=white] (0*1.414,0*1.414) circle (0.3);
\draw[fill=black] (-5*1.414,-5*1.414) circle (0.3);
\draw[fill=black] (-5*1.414,-4*1.414) circle (0.3);
\draw[fill=black] (-5*1.414,-3*1.414) circle (0.3);
\draw[fill=black] (-5*1.414,-2*1.414) circle (0.3);
\draw[fill=black] (-5*1.414,-1*1.414) circle (0.3);
\end{tikzpicture}
  \caption{$\overline{\Pi}(0,10:0,10)$}
  \label{fig:sub2a}
\end{subfigure}
\caption{$(0,10)\otimes(0,10)\in\mathcal{C}(\mathfrak{so}_5,11)$}
\label{fig:test2}
\end{figure}

\par Corollary \ref{goal} applied to (\ref{so5eq2}) gives
\begin{align}
&&\sum_{i=0}^{x-1}[2(i+1)][2(i+2)][2i+3]&\leq[2][m+1][2(m+2)][m+3]\label{eq:1} \\
\Rightarrow\hspace{10 mm}&&\sum_{i=0}^{x-1}(i+1)(i+2)(i+3/2)&\leq4(m+1)(m+2)(m+3),\label{eq:2}
\end{align}
applying Corollary \ref{quantumtrianglecor} to the right-hand side of (\ref{eq:1}) and Lemma \ref{lowerquantumbound} to the left-hand side of (\ref{eq:1}).  Lemma \ref{lowerquantumbound} applies since $m$ even implies $2(2x+2)=2(m+4)<2(k+3)$.  Now we rewrite the right-hand side of (\ref{eq:2}) in terms of $x$ and re-index the left-hand sum, observing each factor on the left-hand side of (\ref{eq:2}) is greater than $i$ to yield
\begin{equation}
\sum_{i=1}^{x}i^3\leq4(2x+3)(2x+4)(2x+5).\label{so5eq2b}
\end{equation}
Using Faulhaber's formula \cite{knuth} on the left-hand side of (\ref{so5eq2b}) produces
\begin{equation*}
\dfrac{1}{4}x^2(x+1)^2\leq4(2x+3)(2x+4)(2x+5)
\end{equation*}
which is true for $x\leq131$, and thus $2x+2=m\leq264$.  From Lemma \ref{twistisone} we have $\theta(0,m)=1$, which implies $(m^2+4m)/(4(k+3))\in\mathbb{Z}$ and thus $k\leq(264^2+4\cdot264)/4-3=17685$.


\subsection{The case $0\neq\ell\leq x<m-2$}\label{soo3}

Set $\lambda:=\ell+m-x+1$ so that $\lambda=x+3$ if $\ell$ is even, and $\lambda=x+2$ if $\ell$ is odd.  We claim if $0\neq\ell\leq x<m-2$, then
\begin{equation}
(\ell+1,m-\lambda)\oplus(\ell-1,m-\lambda+2)\subset(0,\lambda)\otimes(\ell,m).\label{eq:so510}
\end{equation}
The set $\overline{\Pi}(0,\lambda)$, illustrated by example in Figure \ref{fig:test4}, is a square with vertex $(0,-\lambda)$ and its three conjugates under the Weyl group. In particular $\Pi(0,\lambda:\ell,m)$ contains $(\ell+1,m-\lambda)$ and $(\ell-1,m-\lambda+2)$ since $x+2<m$.  The angles formed by $\overline{\Pi}(0,\lambda:\ell,m)$ and $T_2$ are 45 degrees when they exist which implies there is no contribution to the desired fusion coefficients from $\tau_2$, while $\tau_0$ cannot contribute because $(\ell,m)$ does not lie on $T_0$.  Lastly note that $\overline{\Pi}(0,\lambda:\ell,m)$ does not intersect $T_1$ since $x+2<m$ so there can be no contribution from $\tau_1$ either.  Lemma \ref{fusion} then implies containment (\ref{eq:so510}).

\begin{figure}[H]
\centering
\begin{subfigure}{.5\textwidth}
  \centering
\begin{tikzpicture}[scale=0.18]
\draw[fill=black!5,thick] (3.5*1.414,3.5*1.414) -- (3.5*1.414,-3.5*1.414) --  (-3.5*1.414,-3.5*1.414) --  (-3.5*1.414,3.5*1.414) -- cycle;
\foreach \x in {-6,-5,...,6} {
	\foreach \y in {-8,-7,...,8} {
        \node at (\x*1.4142,\y*1.4142) {$\cdot$};
    		}
		};
\foreach \x in {-6,-5,...,6} {
	\foreach \y in {-8,-7,...,8} {
        \node at (\x*1.4142+0.707,\y*1.4142+0.707) {$\cdot$};
    		}
		};
\draw[dotted,thick] (6.5*1.414,0*1.414) -- (-6*1.414,0*1.414);
\draw[dotted,thick] (0*1.414,8.5*1.414) -- (0*1.414,-8*1.414);
\draw[dotted,thick] (-6*1.414,6*1.414) -- (6.5*1.414,-6.5*1.414);
\draw[dotted,thick] (-6*1.414,-6*1.414) -- (6.5*1.414,6.5*1.414);
\draw[fill=white] (4*1.414,7*1.414) circle (0.3);
\draw[fill=black] (-2.5*1.414,-3.5*1.414) circle (0.3);
\draw[fill=black] (-3.5*1.414,-2.5*1.414) circle (0.3);
\end{tikzpicture}
  \caption{$\overline{\Pi}(0,6)$}
  \label{fig:sub1}
\end{subfigure}%
\begin{subfigure}{.5\textwidth}
  \centering
\begin{tikzpicture}[scale=0.18]
\draw[fill=black!5,thick] (3.5*1.414,3.5*1.414) -- (3.5*1.414,-3.5*1.414) --  (-3.5*1.414,-3.5*1.414) --  (-3.5*1.414,3.5*1.414) -- cycle;
\draw[pattern=crosshatch,thick] (-1.5*1.414,3.5*1.414) -- (-1.5*1.414,-0.5*1.414) -- (-0.5*1.414,-1.5*1.414) -- (0.5*1.414,-1.5*1.414) -- (0.5*1.414,-3.5*1.414) -- (-3.5*1.414,-3.5*1.414) --  (-3.5*1.414,3.5*1.414) --  (-2.5*1.414,3.5*1.414) -- cycle;
\draw[fill=black!25,thick] (-1.5*1.414,3.5*1.414) -- (-1.5*1.414,-1.5*1.414) -- (2.5*1.414,-1.5*1.414) -- (3*1.414,-1*1.414) -- cycle;
\draw[fill=black!25,thick] (0.5*1.414,-3.5*1.414) -- (0.5*1.414,-1.5*1.414) -- (2.5*1.414,-1.5*1.414) -- cycle;
\foreach \x in {-6,-5,...,6} {
	\foreach \y in {-8,-7,...,8} {
        \node at (\x*1.4142,\y*1.4142) {$\cdot$};
    		}
		};
\foreach \x in {-6,-5,...,6} {
	\foreach \y in {-8,-7,...,8} {
        \node at (\x*1.4142+0.707,\y*1.4142+0.707) {$\cdot$};
    		}
		};
\draw[dashed,thick] (-4*1.414,-8*1.414) -- node[circle,draw=black,pos = 0.15,inner sep=1pt,fill=white,solid,thin] {\tiny$T_1$}   (-4*1.414,8.5*1.414);
\draw[dashed,thick] (-4*1.414,-8*1.414) -- node[circle,draw=black,pos = 0.9,inner sep=1pt,fill=white,solid,thin] {\tiny$T_2$}   (6.5*1.414,2.5*1.414);
\draw[dashed,thick] (6.5*1.414,-4.5*1.414) -- node[circle,draw=black,pos = 0.1,inner sep=1pt,fill=white,solid,thin] {\tiny$T_0$}   (-6*1.414,8*1.414);
\draw[fill=white] (0*1.414,0*1.414) circle (0.3);
\draw[fill=black] (-2.5*1.414,-3.5*1.414) circle (0.3);
\draw[fill=black] (-3.5*1.414,-2.5*1.414) circle (0.3);
\end{tikzpicture}
  \caption{$\overline{\Pi}(0,6:3,7)$}
  \label{fig:sub4a}
\end{subfigure}
\caption{$(0,6)\otimes(3,7)\in\mathcal{C}(\mathfrak{so}_5,10)$}
\label{fig:test4}
\end{figure}
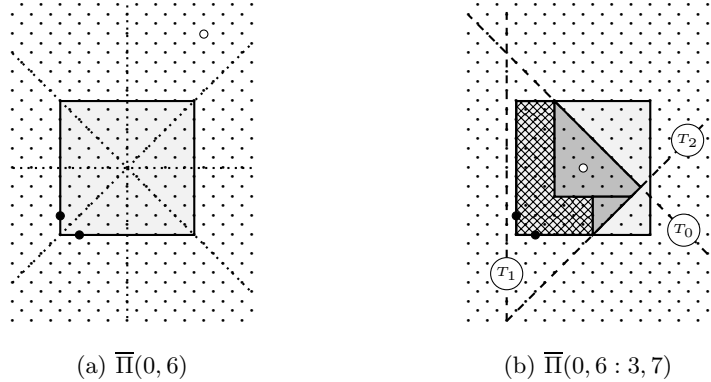

\par Now notice that $(\ell+1,m-\lambda)$ and $(\ell-1,m-\lambda+2)$ are contained in the set of weights $(s,t)\in\Lambda_0$ such that $s+t=x$.  The dimensions of these objects have a clear lower bound.
\begin{lem}\label{so5convex}
If $0\leq x<k/2$, $\dim'(0,x)\leq\dim'(s,x-s)$ for all $0\leq s\leq x$.
\end{lem}
\begin{proof}
With $\kappa:=\pi/(2(k+3))$, define $f(s):=\sin((x-s+1)\kappa)\sin((x+s+3)\kappa)$ and $g(s):=\sin((2s+2)\kappa)$ so that
\begin{equation*}
\dim'(s,x-s)=\sin^{-4}(\kappa)\sin(2(x+2)\kappa)f(s)g(s)
\end{equation*}
as a real function of $s\in[0,x]$ with the constant $\sin^{-4}(\kappa)\sin(2(x+2)\kappa)>0$ since $\kappa$ and $2(x+2)\kappa$ are in the interval $(0,\pi/2)$ for $0\leq x<k/2$.  We will prove our main claim by showing that $(d^2/ds^2)\dim'(s,x-s)<0$ on $[0,x]$ and $\dim'(0,x)\leq\dim'(x,0)$.  It can be easily verified that $f(s)>0$, $g(s)>0$, $g'(s)>0$ and $g''(s)<0$ for $s\in[0,x]$, so we will explicitly compute with $\alpha:=x-s+1$ and $\beta:=x+s+3$ for brevity:
\begin{align*}
&&f'(s)&=\kappa(\sin(\alpha\kappa)\cos(\beta\kappa)-\cos(\alpha\kappa)\sin(\beta\kappa)) \\
	&&&=-\kappa\sin(2(s+2)\kappa) \\
\Rightarrow&&f''(s)&=-2\kappa^2\cos(2(s+2)\kappa).
\end{align*}
The above computations imply $f'(s)<0$ and $f''(s)<0$ for $s\in[0,x]$.  Using the product rule twice implies $(fg)''(s)<0$ and moreover $(d^2/ds^2)\dim'(s,x-s)<0$ since these functions differ by a positive constant factor.

\par Lastly we need to verify $\dim'(0,x)\leq\dim'(x,0)$, or that
\begin{align*}
&&[2][x+1][x+3]&\leq[2x+2][2x+3] \\
\Leftrightarrow&&[x+1][x+3]&\leq\dfrac{[2x+2]}{[2]}[2x+3].
\end{align*}
Note that
\begin{equation*}
\dfrac{[2x+2]}{[2]}=\dfrac{q^{x+1}+q^{-(x+1)}}{q+q^{-1}}[x+1]=\dfrac{\cos\left(\dfrac{(x+1)\pi}{2(k+3)}\right)}{\cos\left(\dfrac{\pi}{2(k+3)}\right)}[x+1]\geq\dfrac{\sqrt{2}}{2}[x+1],
\end{equation*}
because $x+1<(1/2)(k+3)$.  Moreover to complete our proof it would suffice that $[x+3]\leq[2x+3]$.  This inequality is always true because $x+3$ and $2x+3$ are in the interval $(0,k+3)$ and the function $[n]=\sin(n\pi/(2(k+3)))/\sin(\pi/(2(k+3)))$ is strictly increasing for $n\in(0,k+3)$.
\end{proof}

\par Hence when $\ell$ is even, Lemma \ref{so5convex} and Corollary \ref{goal} applied to (\ref{eq:so510}) implies
\begin{align}
&&\dim'(0,x)+\dim'(0,x)&\leq\dim'(\ell+1,m-\lambda)+\dim'(\ell-1,m-\lambda+2) \nonumber\\
&&&\leq\dim'(0,x+3) \label{lastso5}\\
\Rightarrow&&\dim'(0,x)+\dim'(0,x)&\leq[2]([x+1]+3)([2x+4]+6)([x+3]+3)\label{lastso5b}
\end{align}
by applying Lemma \ref{quantumtriangle} to the right-hand side of (\ref{lastso5}).  All terms in (\ref{lastso5b}) have a factor of $[2]$ which we divide out before expanding the product on the right-hand side of (\ref{lastso5b}) and subtracting the leading term (equal to $\dim'(0,x)$) to yield
\begin{equation}
\dfrac{1}{4}(x+1)(x+2)(x+3)\leq6(3x^2+21x+38).\label{lastso52}
\end{equation}
Corollary \ref{quantumtrianglecor} was applied eliminate the quantum analogs on the right-hand side of (\ref{lastso52}) and Lemma \ref{lowerquantumbound} was applied to eliminate the quantum analogs on the left-hand side, which is applicable since $4(x+2)<4(k+3)$ since $x<k/2$.  Inequality (\ref{lastso52}) is true for $x\leq72$, which implies $0<\ell\leq72$ and $74<m\leq145$.  Moreover Lemma \ref{twistisone} implies $k\leq13319$ by maximizing $(2\ell^2+2\ell m+6\ell+m^2+4m)/4-3$ subject to these constraints.  Repeating the above with $\ell$ odd only changes the right-hand side of (\ref{lastso52}) to $12(x+3)^2$, which produces a more restrictive bound on $x$.





\section{Proof of Theorem \ref{thebigone}: $\mathcal{C}(\mathfrak{g}_2,k)$}\label{g2}

\par Let $A$ be a connected \'etale algebra in $\mathcal{C}(\mathfrak{g}_2,k)$ with minimal nontrivial summand $(\ell,m)$ (i.e. $\ell+(3/2)m$ is minimal) and fix $x:=\lceil(1/2)(\ell+(3/2)m)\rceil-1$; the value $x$ is the greatest integer $n$ such that $(n,0)$ satisfies the hypotheses of Lemma \ref{simplefree}.  Similarly one can set $y:=\lceil(1/2)((2/3)\ell+m)\rceil-1$; the value $y$ is the greatest integer $n$ such that $(0,n)$ satisfies the hypotheses of Lemma \ref{simplefree}.  The proof of Theorem \ref{thebigone} will be split into four (clearly exhaustive) cases, illustrated by example in Figure \ref{fig:casesg2}, with varying numbers of subcases for a fixed $x$: $0\leq\ell\leq2$, $3\leq \ell\leq x+3$, $x+3<\ell$ with $m\neq0$, and $m=0$.  

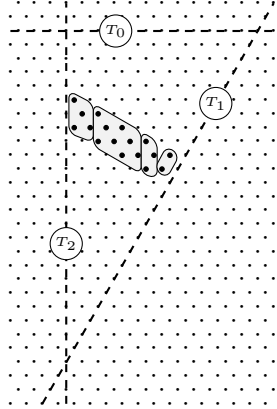
\begin{figure}[H]
\centering
\begin{tikzpicture}[scale=0.15]
\draw[fill=gray!10,rounded corners=1mm] (-0.333*1.414,2.333*2.4495) -- (-0.333*1.414,0.8*2.4495) -- (1.2*1.414,0.466*2.4495)  -- (1.2*1.414,1.866*2.4495) -- cycle;
\draw[fill=gray!10,rounded corners=1mm] (1.2*1.414,1.866*2.4495) -- (1.2*1.414,0.333*2.4495) --  (4.2*1.414,-0.666*2.4495) -- (4.2*1.414,0.866*2.4495)  -- cycle;
\draw[fill=gray!10,rounded corners=1mm] (4.2*1.414,0.866*2.4495) -- (4.2*1.414,-0.666*2.4495) --  (5.2*1.414,-0.833*2.4495) --  (5.2*1.414,0.5*2.4495)  -- cycle;
\draw[fill=gray!10,rounded corners=1mm] (5.8*1.414,0.333*2.4495) -- (6.6*1.414,0*2.4495) -- (5.8*1.414,-0.833*2.4495) --(5.02*1.414,-0.5*2.4495) -- cycle;
\foreach \x in {-4,-3,...,12} {
	\foreach \y in {-9,-8,...,5} {
        \node at (\x*1.414,\y*2.4495) {$\cdot$};
    		}
		};
\foreach \x in {-4,-3,...,12} {
	\foreach \y in {-9,-8,...,5} {
        \node at (\x*1.414+0.5*1.414,1.225+\y*2.4495) {$\cdot$};
    		}
		};

\draw[fill=black] (6*1.414,0*2.4495) circle (0.2);
\draw[fill=black] (5.5*1.414,-0.5*2.4495) circle (0.2);

\draw[fill=black] (0*1.414,2*2.4495) circle (0.2);
\draw[fill=black] (0*1.414,1*2.4495) circle (0.2);
\draw[fill=black] (0.5*1.414,1.5*2.4495) circle (0.2);
\draw[fill=black] (1*1.414,1*2.4495) circle (0.2);

\draw[fill=black] (1.5*1.414,1.5*2.4495) circle (0.2);
\draw[fill=black] (2*1.414,1*2.4495) circle (0.2);
\draw[fill=black] (1.5*1.414,0.5*2.4495) circle (0.2);
\draw[fill=black] (2.5*1.414,0.5*2.4495) circle (0.2);
\draw[fill=black] (3.5*1.414,0.5*2.4495) circle (0.2);
\draw[fill=black] (3*1.414,0*2.4495) circle (0.2);
\draw[fill=black] (3*1.414,1*2.4495) circle (0.2);
\draw[fill=black] (4*1.414,0*2.4495) circle (0.2);
\draw[fill=black] (4.5*1.414,-0.5*2.4495) circle (0.2);
\draw[fill=black] (4.5*1.414,0.5*2.4495) circle (0.2);
\draw[fill=black] (5*1.414,0*2.4495) circle (0.2);

\draw[thick,dashed] (-2*1.414,-9*2.4495) -- node[circle,draw=black,pos = 0.75,inner sep=1pt,fill=white,solid,thin] {\tiny$T_1$}  (12.5*1.414,5.5*2.4495);
\draw[thick,dashed] (-4*1.414,4.5*2.4495) --  node[circle,draw=black,pos = 0.4,inner sep=1pt,fill=white,solid,thin] {\tiny$T_0$} (12.5*1.414,4.5*2.4495);
\draw[thick,dashed] (-0.5*1.414,-9*2.4495) --  node[circle,draw=black,pos = 0.4,inner sep=1pt,fill=white,solid,thin] {\tiny$T_2$} (-0.5*1.414,5.5*2.4495);
\end{tikzpicture}
\caption{Possible $(\ell,m)$ when $k=20$ and $x=5$}
\label{fig:casesg2}
\end{figure}


\subsection{The case $0\leq\ell\leq2$}\label{g2emm2}

\subsubsection{The subcase $\ell=0$}\label{g2ellequalszero}

\par We will employ the same strategy as Section \ref{sog1}.  Recall $y=\lceil m/2\rceil-1$ if $\ell=0$ and set $\lambda:=m-y+1$ so that $\lambda=y+3$ if $m$ is even, and $\lambda=y+2$ if $m$ is odd.  We claim for $4<m\leq k/2$,
\begin{equation}
(0,y-1)\oplus(3,y-2)\subset(0,\lambda)\otimes(0,m).\label{eq:g2boi133}
\end{equation}
The set $\overline{\Pi}(0,\lambda)$, illustrated by example in Figure \ref{fig:g21}, (refer to Section \ref{sec:geometry} for descriptions of the notation and visualization used), is a hexagon with vertex $(0,-\lambda)$ and its five conjugates under the Weyl group.  In particular $\Pi(0,\lambda:0,m)$ contains $(0,y-1)$ and $(3,y-2)$ since $m>4$.  There is no contribution to $N_{(0,\lambda),(0,m)}^{(0,y-1)}$ or $N_{(0,\lambda),(0,m)}^{(3,y-2)}$ from $\tau_2$ because the angles formed by $\overline{\Pi}(0,\lambda:0,m)$ and $T_2$ are 60 degrees and there is no contribution from $\tau_1$ because the angles formed by $\overline{\Pi}(0,\lambda:0,m)$ and $T_1$ (when they exist) are 30 degrees.  There is no contribution from $\tau_0$ because $(0,m)$ does not lie on $T_0$.  Lemma \ref{fusion} then implies containment (\ref{eq:g2boi133}).
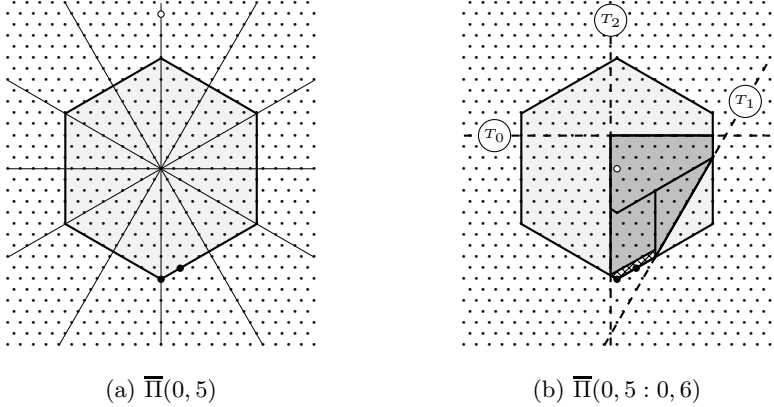
\begin{figure}[H]
\centering
\begin{subfigure}{.5\textwidth}
  \centering
\begin{tikzpicture}[scale=0.12]
\draw[fill=black!5,thick] (0*1.414,5*2.4495) -- (7.5*1.414,2.5*2.4495) -- (7.5*1.414,-2.5*2.4495) -- (0*1.414,-5*2.4495) -- (-7.5*1.414,-2.5*2.4495) -- (-7.5*1.414,2.5*2.4495) -- cycle;
\foreach \x in {-12,-11,...,12} {
	\foreach \y in {-8,-7,...,7} {
        \node at (\x*1.414,\y*2.4495) {$\cdot$};
    		}
		};
\foreach \x in {-12,-11,...,11} {
	\foreach \y in {-8,-7,...,7} {
        \node at (\x*1.414+0.5*1.414,1.225+\y*2.4495) {$\cdot$};
    		}
		};
\draw[very thin] (-12*1.414,0*2.4495) -- (12*1.414,0*2.4495);
\draw[very thin] (0*1.414,7.5*2.4495) -- (0*1.414,-8*2.4495);
\draw[very thin] (-12*1.414,4*2.4495) -- (12*1.414,-4*2.4495);
\draw[very thin] (-12*1.414,-4*2.4495) -- (12*1.414,4*2.4495);
\draw[very thin] (7.5*1.414,7.5*2.4495) -- (-8*1.414,-8*2.4495);
\draw[very thin] (8*1.414,-8*2.4495) -- (-7.5*1.414,7.5*2.4495);
\draw[fill=white] (0*1.414,7*2.4495) circle (0.35);
\draw[fill=black] (0*1.414,-5*2.4495) circle (0.35);
\draw[fill=black] (1.5*1.414,-4.5*2.4495) circle (0.35);
\end{tikzpicture}
  \caption{$\overline{\Pi}(0,5)$}
  \label{fig:g21a}
\end{subfigure}%
\begin{subfigure}{.5\textwidth}
  \centering
\begin{tikzpicture}[scale=0.12]
\draw[fill=black!5,thick] (0*1.414,5*2.4495) -- (7.5*1.414,2.5*2.4495) -- (7.5*1.414,-2.5*2.4495) -- (0*1.414,-5*2.4495) -- (-7.5*1.414,-2.5*2.4495) -- (-7.5*1.414,2.5*2.4495) -- cycle;
\draw[pattern=crosshatch,thick] (-0.5*1.414,1.5*2.4495) --  (7.5*1.414,1.5*2.4495) -- (7.5*1.414,0.5*2.4495) -- (3*1.414,-4*2.4495) -- (0*1.414,-5*2.4495)-- (-0.5*1.414,-4.8*2.4495) -- cycle;
\draw[fill=black!25,thick] (-0.5*1.414,1.5*2.4495) -- (6.5*1.414,1.5*2.4495) -- (6.5*1.414,-0.5*2.4495)-- (3.5*1.414,-3.5*2.4495) -- (-0.5*1.414,-4.8*2.4495) -- cycle;
\draw[fill=black!25,thick] (-0.5*1.414,1.5*2.4495) -- (-0.5*1.414,-1.8*2.4495) -- (0*1.414,-2*2.4495) -- (7.5*1.414,0.5*2.4495) -- (7.5*1.414,1.5*2.4495) -- cycle;
\draw[fill=black!25,thick] (3*1.414,-4*2.4495) -- (3*1.414,-1*2.4495) -- (7.5*1.414,0.5*2.4495) -- cycle;
\foreach \x in {-12,-11,...,12} {
	\foreach \y in {-8,-7,...,7} {
        \node at (\x*1.414,\y*2.4495) {$\cdot$};
    		}
		};
\foreach \x in {-12,-11,...,11} {
	\foreach \y in {-8,-7,...,7} {
        \node at (\x*1.414+0.5*1.414,1.225+\y*2.4495) {$\cdot$};
    		}
		};
\draw[thick,dashed] (-1*1.414,-8*2.4495) -- node[circle,draw=black,pos = 0.85,inner sep=1pt,fill=white,solid,thin] {\tiny$T_1$}  (12*1.414,5*2.4495);
\draw[thick,dashed] (-12*1.414,1.5*2.4495) --  node[circle,draw=black,pos = 0.1,inner sep=1pt,fill=white,solid,thin] {\tiny$T_0$} (12*1.414,1.5*2.4495);
\draw[thick,dashed] (-0.5*1.414,-8*2.4495) --  node[circle,draw=black,pos = 0.95,inner sep=1pt,fill=white,solid,thin] {\tiny$T_2$} (-0.5*1.414,7.5*2.4495);
\draw[fill=white] (0*1.414,0*2.4495) circle (0.35);
\draw[fill=black] (0*1.414,-5*2.4495) circle (0.35);
\draw[fill=black] (1.5*1.414,-4.5*2.4495) circle (0.35);
\end{tikzpicture}
  \caption{$\overline{\Pi}(0,5:0,6)$}
  \label{fig:g21b}
\end{subfigure}
\caption{$(0,5)\otimes(0,6)\in\mathcal{C}(\mathfrak{g}_2,18)$}
\label{fig:g21}
\end{figure}

\par If $m$ is even, Corollary \ref{goal} applied to (\ref{eq:g2boi133}) gives
\begin{align}
&\dim'(0,y-1)+\dim'(3,y-2)\label{eq:g2boi}  \\
&\leq[3y+12][3y+15][6y+27][3y+13][3y+14]\label{eq:g2boi2} \\
&\leq([3y]+12)([3y+3]+12)([6y+3]+24)([3y+1]+12)([3y+2]+12)\label{eq:g2boi3}
\end{align}
where (\ref{eq:g2boi3}) is gained by applying Corollary \ref{quantumtrianglecor} to (\ref{eq:g2boi2}).  Expanding the product in (\ref{eq:g2boi3}) and subtracting the leading term (equal to $\dim'(0,y-1)$) yields
\begin{equation}
\dfrac{27}{8}(y-1)(y+3)(y+1)(3y+1)(3y+5)\leq3240(3y^4+30y^3+136y^2+305y+273)\nonumber
\end{equation}
by applying Lemma \ref{lowerquantumbound} to the factors of $\dim'(3,y-2)$ in (\ref{eq:g2boi}), which is true for even $m$ with $y\leq324$ or likewise $m\leq650$.  From Lemma \ref{twistisone}, $\theta(0,m)=1$ which implies $(3m^2+9m)/(3(k+4))\in\mathbb{Z}$ and moreover $k\leq(1/3)(3(650)^2+9(650))-4=424446$.

\par By repeating the above argument with $m$ odd we obtain the inequality
\begin{equation}
\dfrac{27}{8}(y-1)(y+3)(y+1)(3y+1)(3y+5)<810(9y^4+72y^3+255y^2+444y+308)
\nonumber\end{equation}
which is true for $y\leq242$ which evidently yields a stricter bound on $k$.


\subsubsection{The subcase $\ell=1$}

\par The strategy is identical to Section \ref{g2ellequalszero}, except with $\lambda:=m-y+1$ we claim
\begin{equation*}
(0,y-1)\oplus(3,y-2)\subset(1,\lambda)\otimes(0,m),\label{eq:g2boi1a}
\end{equation*}
and we omit the redundant arguments for both this containment and to produce the following inequalities, based on $m$ being even or odd, respectively:
\begin{align*}
\dfrac{27}{32}y(y+1)(2y+1)(3y+1)(3y+2)&\leq324(54y^4+613y^3+2861y^2+\beta_1y+\beta_2) \\
\dfrac{27}{32}y(y+1)(2y+1)(3y+1)(3y+2)&\leq1620(y+3)(9y^3+65y^2+183y+191)
\end{align*}
where $\beta_1=6427$ and $\beta_2=5725$ for display purposes.  The first inequality is true for even $m$ with $y\leq1160$ and the second for odd $m$ with $y\leq967$, hence $m\leq2322$ and moreover $k\leq(1^2+3(1)(2322)+5(1)+3(2322)^2+9(2322))/3-4=5400970$.


\subsubsection{The subcase $\ell=2$}

\par The strategy is identical to Section \ref{g2ellequalszero}, except with $\lambda:=m-y+1$ we claim
\begin{equation*}
(0,y-1)\oplus(3,y-2)\subset(2,\lambda)\otimes(0,m),\label{eq:g2boi1aa}
\end{equation*}
and so we omit the redundant argument to produce the following inequalities, based on $m$ being even or odd, respectively:
\begin{align*}
\dfrac{27}{32}y(y+1)(2y+1)(3y+1)(3y+2)&\leq81(399y^4+5171y^3+\beta_1y^2+\beta_2y+\beta_3) \\
\dfrac{27}{32}y(y+1)(2y+1)(3y+1)(3y+2)&\leq2835(y+3)(9y^3+73y^2+234y+278)
\end{align*}
where $\beta_1=28239$, $\beta_2=74821$, and $\beta_3=78570$ for display purposes.  The first inequality is true for even $m$ with $y\leq2138$ and the second for odd $m$ with $y\leq1688$, hence $m\leq4272$ and moreover $k\leq(2^2+3(2)(4272)+5(2)+3(4272)^2+9(4272))/3-4<18271135$.


\subsection{The case $m=0$}\label{g2emm}

Recall $x=\lceil\ell/2\rceil-1$ if $m=0$.  Set $\lambda:=\ell-x+1$ so that $\lambda=x+3$ if $\ell$ is even and $\lambda=x+2$ if $\ell$ is odd.  We claim that for $4<\ell\leq k$,
\begin{equation}
(x-1,0)\oplus(x-2,1)\subset(\lambda,0)\otimes(\ell,0).\label{eq:g2man1}
\end{equation}
The set $\overline{\Pi}(\lambda,0)$, illustrated by example in Figure \ref{fig:g21e}, is a hexagon with vertex $(-\lambda,0)$ and its five conjugates under the Weyl group.  In particular $\Pi(\lambda,0:\ell,0)$ contains $(x-1,0)$ and $(x-2,1)$ provided $\ell>4$.  The angles formed by $\overline{\Pi}(\lambda,0:\ell,0)$ and $T_1$ are 30 degrees and the angles formed by $\overline{\Pi}(\lambda,0:\ell,0)$ and $T_0,T_2$ are 60 degrees, ensuring there can be no contribution from $\tau_0,\tau_1,\tau_2$.   Lemma \ref{fusion} then implies containment (\ref{eq:g2man1}).

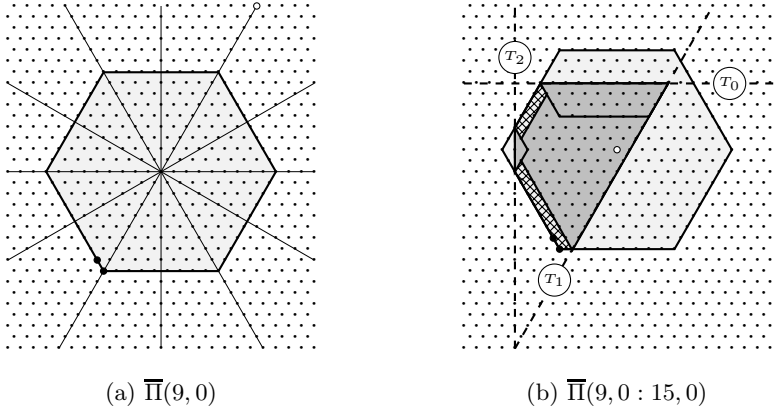
\begin{figure}[H]
\centering
\begin{subfigure}{.5\textwidth}
  \centering
\begin{tikzpicture}[scale=0.12]
\draw[fill=black!5,thick] (4.5*1.414,4.5*2.4495) -- (9*1.414,0*2.4495) -- (4.5*1.414,-4.5*2.4495) -- (-4.5*1.414,-4.5*2.4495) -- (-9*1.414,0*2.4495) -- (-4.5*1.414,4.5*2.4495) -- cycle;
\foreach \x in {-12,-11,...,12} {
	\foreach \y in {-8,-7,...,7} {
        \node at (\x*1.414,\y*2.4495) {$\cdot$};
    		}
		};
\foreach \x in {-12,-11,...,11} {
	\foreach \y in {-8,-7,...,7} {
        \node at (\x*1.414+0.5*1.414,1.225+\y*2.4495) {$\cdot$};
    		}
		};
\draw[very thin] (-12*1.414,0*2.4495) -- (12*1.414,0*2.4495);
\draw[very thin] (0*1.414,7.5*2.4495) -- (0*1.414,-8*2.4495);
\draw[very thin] (-12*1.414,4*2.4495) -- (12*1.414,-4*2.4495);
\draw[very thin] (-12*1.414,-4*2.4495) -- (12*1.414,4*2.4495);
\draw[very thin] (7.5*1.414,7.5*2.4495) -- (-8*1.414,-8*2.4495);
\draw[very thin] (8*1.414,-8*2.4495) -- (-7.5*1.414,7.5*2.4495);
\draw[fill=white] (7.5*1.414,7.5*2.4495) circle (0.35);
\draw[fill=black] (-4.5*1.414,-4.5*2.4495) circle (0.35);
\draw[fill=black] (-5*1.414,-4*2.4495) circle (0.35);
\end{tikzpicture}
  \caption{$\overline{\Pi}(9,0)$}
  \label{fig:g21a}
\end{subfigure}%
\begin{subfigure}{.5\textwidth}
  \centering
\begin{tikzpicture}[scale=0.12]
\draw[fill=black!5,thick] (4.5*1.414,4.5*2.4495) -- (9*1.414,0*2.4495) -- (4.5*1.414,-4.5*2.4495) -- (-4.5*1.414,-4.5*2.4495) -- (-9*1.414,0*2.4495) -- (-4.5*1.414,4.5*2.4495) -- cycle;
\draw[pattern=crosshatch,thick] (-3.5*1.414,-4.5*2.4495) -- (-4.5*1.414,-4.5*2.4495) -- (-8*1.414,-1*2.4495) -- (-8*1.414,1*2.4495)  -- (-6*1.414,3*2.4495) -- (4*1.414,3*2.4495) -- cycle;
\draw[fill=black!25,thick] (-8*1.414,0*2.4495) -- (-5*1.414,3*2.4495) -- (4*1.414,3*2.4495) -- (-3.5*1.414,-4.55*2.4495) -- cycle;
\draw[fill=black!25,thick] (-8*1.414,-1*2.4495) -- (-8*1.414,1*2.4495) -- (-7*1.414,0*2.4495) -- cycle;
\draw[fill=black!25,thick] (-6*1.414,3*2.4495) -- (-4.5*1.414,1.5*2.4495) -- (2.5*1.414,1.5*2.4495) -- (4*1.414,3*2.4495) -- cycle;
\foreach \x in {-12,-11,...,12} {
	\foreach \y in {-9,-8,...,6} {
        \node at (\x*1.414,\y*2.4495) {$\cdot$};
    		}
		};
\foreach \x in {-12,-11,...,11} {
	\foreach \y in {-9,-8,...,6} {
        \node at (\x*1.414+0.5*1.414,1.225+\y*2.4495) {$\cdot$};
    		}
		};
\draw[thick,dashed] (-8*1.414,-9*2.4495) -- node[circle,draw=black,pos = 0.85,inner sep=1pt,fill=white,solid,thin] {\tiny$T_2$}  (-8*1.414,6.5*2.4495);
\draw[thick,dashed] (-12*1.414,3*2.4495) -- node[circle,draw=black,pos = 0.85,inner sep=1pt,fill=white,solid,thin] {\tiny$T_0$} (12.5*1.414,3*2.4495);
\draw[thick,dashed] (-8*1.414,-9*2.4495) --node[circle,draw=black,pos = 0.2,inner sep=1pt,fill=white,solid,thin] {\tiny$T_1$}  (7.5*1.414,6.5*2.4495);
\draw[fill=white] (0*1.414,0*2.4495) circle (0.35);
\draw[fill=black] (-4.5*1.414,-4.5*2.4495) circle (0.35);
\draw[fill=black] (-5*1.414,-4*2.4495) circle (0.35);
\end{tikzpicture}
  \caption{$\overline{\Pi}(9,0:15,0)$}
  \label{fig:g21b}
\end{subfigure}
\caption{$(9,0)\otimes(15,0)\in\mathcal{C}(\mathfrak{g}_2,20)$}
\label{fig:g21e}
\end{figure}

If $m$ is even, Corollary \ref{goal} applied to (\ref{eq:g2man1}) gives
\begin{align}
&\dim'(x-1,0)+\dim'(x-2,1) \label{eq:g2man}\\
&\leq[x+4][3][3x+15][3x+18][x+7][2x+11] \label{eq:g2man2}\\
&\leq([x]+4)[3]([3x+3]+12)([3x+6]+12)([x+3]+4)([2x+3]+8)\label{eq:g2man3}
\end{align}
where (\ref{eq:g2man3}) is gained by applying Corollary \ref{quantumtrianglecor} to (\ref{eq:g2man2}).  Expanding the product in (\ref{eq:g2man3}) and subtracting the leading term (which is equal to $\dim'(x-1,0)$) yields
\begin{equation}
\dfrac{27}{16}(x-1)(x+2)(x+3)(x+5)(x+2)\leq1080(x^4+14x^3+80x^2+217x+231)
\end{equation}
by applying Lemma \ref{lowerquantumbound} to the factors of $\dim'(x-2,1)$ in (\ref{eq:g2man}), which is true for even $x\leq642$ or likewise $\ell\leq1286$.  From Lemma \ref{twistisone} we know $\theta(\ell,0)=1$ which implies $(\ell^2+5\ell)/(3(k+4))\in\mathbb{Z}$ and with the proven bound on $\ell$, $k\leq(1/3)((1286)^2+5(1286))-4=1660214/3<553405$.

\par If $m$ is odd, the above process yields the inequality
\begin{equation*}
\dfrac{27}{16}(x-1)(x+2)(x+3)(x+5)(x+2)\leq810(x+3)^2(x^2+6x+12)
\end{equation*}
which is true for $x\leq481$ which evidently yields a stricter bound on $k$.


\subsection{The case $3\leq\ell\leq x+3$}


\subsubsection{The subcase $\ell\equiv0\pmod{3}$}\label{sec:ellmodzero}

\par Recall $y=\lceil(1/2)((2/3)\ell+m)\rceil-1$ and set $\lambda:=(2/3)\ell+m-y$ so that $\lambda=y+2$ if $m$ is even and $\lambda=y+1$ if $m$ is odd.  We claim
\begin{equation}
(0,y)\oplus(3,y-2)\subset\left(0,\lambda\right)\otimes(\ell,m).\label{eq:g2biz1}
\end{equation}
The set $\overline{\Pi}(0,\lambda)$, illustrated by example in Figure \ref{fig:g21d}, is a hexagon with vertex $(0,-\lambda)$ and its five conjugates under the Weyl group.  In particular $\Pi(0,\lambda:\ell,m)$ contains $(0,y)$ and $(3,y-2)$.  To see this, $\Pi(0,\lambda:\ell,m)$ contains more generally all $(\ell-3i,m-\lambda+2i)$ for all $0\leq i\leq(1/3)\ell$.  The angles formed by $\overline{\Pi}(0,\lambda:\ell,m)$ and $T_1$ are 30 degrees and the angles formed by $\overline{\Pi}(0,\lambda:\ell,m)$ and $T_2$ are 60 degrees, implying there are no contributions from $\tau_1,\tau_2$.  The angles formed by $\overline{\Pi}(0,\lambda:\ell,m)$ and $T_0$ are 90 (or 30) degrees when they exist, but since $(0,m)$ does not lie on $T_0$ there is no contribution from $\tau_0$.  Lemma \ref{fusion} then implies containment (\ref{eq:g2biz1}).

\begin{figure}[H]
\centering
\begin{subfigure}{.5\textwidth}
  \centering
\begin{tikzpicture}[scale=0.12]
\draw[fill=black!5,thick] (0*1.414,4*2.4495) -- (6*1.414,2*2.4495) -- (6*1.414,-2*2.4495) -- (0*1.414,-4*2.4495) -- (-6*1.414,-2*2.4495) -- (-6*1.414,2*2.4495) -- cycle;
\foreach \x in {-12,-11,...,12} {
	\foreach \y in {-8,-7,...,7} {
        \node at (\x*1.414,\y*2.4495) {$\cdot$};
    		}
		};
\foreach \x in {-12,-11,...,11} {
	\foreach \y in {-8,-7,...,7} {
        \node at (\x*1.414+0.5*1.414,1.225+\y*2.4495) {$\cdot$};
    		}
		};
\draw[very thin] (-12*1.414,0*2.4495) -- (12*1.414,0*2.4495);
\draw[very thin] (0*1.414,7.5*2.4495) -- (0*1.414,-8*2.4495);
\draw[very thin] (-12*1.414,4*2.4495) -- (12*1.414,-4*2.4495);
\draw[very thin] (-12*1.414,-4*2.4495) -- (12*1.414,4*2.4495);
\draw[very thin] (7.5*1.414,7.5*2.4495) -- (-8*1.414,-8*2.4495);
\draw[very thin] (8*1.414,-8*2.4495) -- (-7.5*1.414,7.5*2.4495);
\draw[fill=white] (1.5*1.414,5.5*2.4495) circle (0.35);
\draw[fill=black] (0*1.414,-4*2.4495) circle (0.35);
\draw[fill=black] (-1.5*1.414,-3.5*2.4495) circle (0.35);
\end{tikzpicture}
  \caption{$\overline{\Pi}(0,4)$}
  \label{fig:g21a}
\end{subfigure}%
\begin{subfigure}{.5\textwidth}
  \centering
\begin{tikzpicture}[scale=0.12]
\draw[fill=black!5,thick] (0*1.414,4*2.4495) -- (6*1.414,2*2.4495) -- (6*1.414,-2*2.4495) -- (0*1.414,-4*2.4495) -- (-6*1.414,-2*2.4495) -- (-6*1.414,2*2.4495) -- cycle;
\draw[pattern=crosshatch,thick] (-2*1.414,2.5*2.4495) --  (4.5*1.414,2.5*2.4495) -- (6*1.414,2*2.4495) -- (6*1.414,1*2.4495) -- (1.5*1.414,-3.5*2.4495) -- (0*1.414,-4*2.4495) -- (-2*1.414,-3.33*2.4495)  -- cycle;
\draw[fill=black!25,thick] (-2*1.414,2.5*2.4495) -- (2*1.414,2.5*2.4495) -- (2*1.414,-2*2.4495) -- (-2*1.414,-3.33*2.4495) -- cycle;
\draw[fill=black!25,thick] (-2*1.414,2.5*2.4495) -- (4.5*1.414,2.5*2.4495)  -- (0*1.414,1*2.4495)-- (-2*1.414,1.66*2.4495) -- cycle;
\draw[fill=black!25,thick] (1.5*1.414,-3.5*2.4495) -- (1.5*1.414,-0.5*2.4495) -- (6*1.414,1*2.4495) -- cycle;
\foreach \x in {-12,-11,...,12} {
	\foreach \y in {-8,-7,...,7} {
        \node at (\x*1.414,\y*2.4495) {$\cdot$};
    		}
		};
\foreach \x in {-12,-11,...,11} {
	\foreach \y in {-8,-7,...,7} {
        \node at (\x*1.414+0.5*1.414,1.225+\y*2.4495) {$\cdot$};
    		}
		};
\draw[thick,dashed] (-3*1.414,-8*2.4495) -- node[circle,draw=black,pos = 0.85,inner sep=1pt,fill=white,solid,thin] {\tiny$T_1$}  (12*1.414,7*2.4495);
\draw[thick,dashed] (-12*1.414,2.5*2.4495) -- node[circle,draw=black,pos = 0.15,inner sep=1pt,fill=white,solid,thin] {\tiny$T_0$}  (12*1.414,2.5*2.4495);
\draw[thick,dashed] (-2*1.414,-8*2.4495) -- node[circle,draw=black,pos = 0.9,inner sep=1pt,fill=white,solid,thin] {\tiny$T_2$}  (-2*1.414,7.5*2.4495);
\draw[fill=white] (0*1.414,0*2.4495) circle (0.35);
\draw[fill=black] (0*1.414,-4*2.4495) circle (0.35);
\draw[fill=black] (-1.5*1.414,-3.5*2.4495) circle (0.35);
\end{tikzpicture}
  \caption{$\overline{\Pi}(0,4:3,4)$}
  \label{fig:g21b}
\end{subfigure}
\caption{$(0,4)\otimes(3,4)\in\mathcal{C}(\mathfrak{g}_2,15)$}
\label{fig:g21d}
\end{figure}
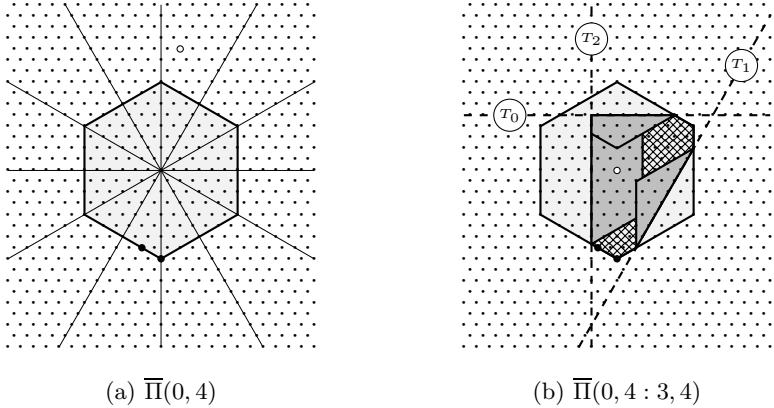

\par If $m$ is even, Corollary \ref{goal} applied to (\ref{eq:g2biz1}) gives
\begin{align}
&\dim'(0,y)+\dim'(3,y-2)\label{eq:g2biz} \\
&\leq[3y+9][3y+12][6y+21][3y+10][3y+11]\label{eq:g2biz2}\\
&\leq([3y+3]+6)([3y+6]+6)([6y+9]+18)([3y+4]+6)([3y+5]+6)\label{eq:g2biz3}
\end{align}
by applying Corollary \ref{quantumtrianglecor} to (\ref{eq:g2biz2}).  Expanding the product in (\ref{eq:g2biz3}) and subtracting the leading term (equal to $\dim'(0,y)$) yields
\begin{equation}
\dfrac{27}{8}(y-1)(y+3)(y+1)(3y+1)(3y+5)\leq1620(y^2+5y+8)(3y^2+15y+19)
\end{equation}
which is true for even $(2/3)\ell+m$ with $y\leq164$.  This bound implies $\ell+(3/2)m\leq495$.  From Lemma \ref{twistisone} we know $\theta(\ell,m)=1$ which implies $k\leq(1/3)(\ell^2+3\ell m+5\ell+3m^2+9m)-4$; and for $\ell+(3/2)m\leq495$ we have $k\leq109886$.  As in Sections \ref{g2emm2} and \ref{g2emm}, the case in which $\ell+(3/2)m$ is odd leads to a stricter bound on $k$ by this method.


\subsubsection{The subcase $\ell\equiv1\pmod{3}$}\label{subsubsub}

\par With $y=\lceil(1/2)((2/3)\ell+m)\rceil-1$, we set $\lambda:=(2/3)(\ell-1)+m-y$.  This implies $\lambda=y$ if $m$ is even and $\lambda=y+1$ if $m$ is odd.  We claim
\begin{equation}
(0,y)\oplus(3,y-2)\subset\left(1,\lambda\right)\otimes(\ell,m)\label{eq:g2dom1}
\end{equation}
and we omit the argument for this containment as it is identical to that of Section \ref{sec:ellmodzero}.

\par If $m$ is odd, Corollary \ref{goal} applied to (\ref{eq:g2dom1}) gives
\begin{align}
&\hspace{-4 mm}\dim'(0,y)+\dim'(3,y-2) \\
&\hspace{-4 mm}\leq[2][3y+6][3y+12][6y+18][3y+8][3y+10]\label{eq:g2dom2}\\
&\hspace{-4 mm}\leq[4]([3y-3]+9)([3y+9]+3)([6y+6]+12)([3y+1]+7)([3y+5]+5)\label{eq:g2dom3}
\end{align}
by applying Corollary \ref{quantumtrianglecor} to the right-hand side of (\ref{eq:g2dom2}).  Expanding the product on the right-hand side of (\ref{eq:g2dom3}) and subtracting the leading term (equal to $\dim'(0,y)$) yields
\begin{equation}
\dfrac{27}{8}(y-1)(y+3)(y+1)(3y+1)(3y+5)\leq648(y+3)(12y^3-20y^2-282y-425)
\end{equation}
which is true for $y\leq252$.  Hence we have $(2/3)\ell+m\leq1288$ and furthermore $\ell+(3/2)m\leq1933$.  The level $k$ is bounded under these constraints by $k\leq1664094$.   As in Sections \ref{g2emm2} and \ref{g2emm}, the case in which $m$ is even leads to a stricter bound on $k$ by this method.


\subsubsection{The subcase $\ell\equiv2\pmod{3}$}

\par  With $y=\lceil(1/2)((2/3)\ell+m)\rceil-1$, we set $\lambda:=(2/3)(\ell-2)+m-y$.  This implies $\lambda=y$ if $m$ is even and $\lambda=y-1$ if $m$ is odd.  We claim
\begin{equation}
(0,y)\oplus(3,y)\subset\left(2,\lambda\right)\otimes(\ell,m)\label{eq:g2doom1}
\end{equation}
and we omit the argument for this containment as it is identical to that of Section \ref{sec:ellmodzero}.

\par If $m$ is even, Corollary \ref{goal} applied to (\ref{eq:g2doom1}) gives
\begin{align}
&\dim'(0,y)+\dim'(3,y-2)\label{eq:g2doom} \\
&\leq[3][3y+3][3y+12][6y+15][3y+6][3y+9]\label{eq:g2doom2}\\
&\leq[4]([3y-3]+6)([3y+9]+3)([6y+6]+9)([3y+1]+5)([3y+5]+4)\label{eq:g2doom3}
\end{align}
by applying Corollary \ref{quantumtrianglecor} to (\ref{eq:g2dom2}).  Expanding the product in (\ref{eq:g2dom3}) and subtracting the leading term (which is equal to $\dim(0,x)$) yields
\begin{equation}
\dfrac{27}{8}(y-1)(y+3)(y+1)(3y+1)(3y+5)\leq540(y+3)(y+1)(27y^2+88y+74)
\end{equation}
which is true for $y\leq962$, hence $(2/3)\ell+m\leq1926$ and moreover $\ell+(3/2)m\leq2889$. This produces a bound of $k\leq3715250$.  As in Sections \ref{g2emm2} and \ref{g2emm}, the case in which $m$ is odd leads to a stricter bound on $k$ by this method.


\subsection{The case $x+3<\ell$ and $m\neq0$}

\par We will employ a similar strategy to Section \ref{soo3}.  We first claim that if $x+3<\ell$, then for some $x+1\leq\lambda\leq x+3$, $(\lambda,0)\otimes(\ell,m)$ contains two summands $(s,t)$ such that $s+(3/2)t=x$, depending on the parity of $\ell$ and remainder of $m$ modulo 4.  We will provide proof of this claim in the most extreme case $\ell$ is even and $4\mid m$, using $\lambda=x+3$, leaving the other near identical cases to the reader (geometrically this fact should be evident).  The only changes in each case are due to the slight differences caused by the ceiling function in the definition of $x$.  Note that under our current assumptions $x=(1/2)\ell+(3/4)m-1$.
\par The set $\overline{\Pi}(\lambda,0)$, illustrated by example in Figure \ref{fig:g2222a}, is a hexagon with vertex $(-\lambda,0)$ and its five conjugates under the Weyl group.  In particular $\Pi(\lambda,0:\ell,m)$ contains $(\ell-\lambda-2,m+2)$ and $(\ell-\lambda+4,m-2)$.  The angles formed by $\overline{\Pi}(\lambda,0:\ell,m)$ and $T_1$ are 30 degrees and the angles formed by $\overline{\Pi}(\lambda,0:\ell,0)$ and $T_0,T_2$ are 60 degrees, ensuring there can be no contribution from $\tau_0,\tau_1,\tau_2$.   Lemma \ref{fusion} then implies the fusion coefficients $N^{(\ell-\lambda-2,m+2)}_{(\lambda,0),(\ell,m)}$ and $N^{(\ell-\lambda+4,m-2)}_{(\lambda,0),(\ell,m)}$ are nonzero as desired, provided $(\ell-\lambda+2,m-2)$ and $(\ell-\lambda-2,m+2)$ are in $\Lambda_0$ which is assured since $\ell>x+3$ and $m\geq4$ under our current assumptions.  It remains to note that since $\ell$ is even and $4\mid m$, then
\begin{align*}
(\ell-\lambda+4)+\dfrac{3}{2}(m-2)&=2\left(\dfrac{1}{2}\ell+\dfrac{3}{4}m-1\right)+3-\lambda \\
							&=2x+3-(x+3)=x
\end{align*}
as desired.  Similarly it can be checked $(\ell-\lambda-2)+(3/2)(m+2)=x$ as well.

\begin{figure}[H]
\centering
\begin{subfigure}{.5\textwidth}
  \centering
\begin{tikzpicture}[scale=0.12]
\draw[fill=black!5,thick] (6*1.414,6*2.4495) --  (12*1.414,0*2.4495) -- (6*1.414,-6*2.4495) -- (-6*1.414,-6*2.4495) -- (-12*1.414,0*2.4495) -- (-6*1.414,6*2.4495) -- cycle;
\foreach \x in {-13,-12,...,13} {
	\foreach \y in {-8,-7,...,11} {
        \node at (\x*1.414,\y*2.4495) {$\cdot$};
    		}
		};
\foreach \x in {-13,-12,...,12} {
	\foreach \y in {-8,-7,...,11} {
        \node at (\x*1.414+0.5*1.414,1.225+\y*2.4495) {$\cdot$};
    		}
		};
\draw[very thin] (-13*1.414,0*2.4495) -- (13*1.414,0*2.4495);
\draw[very thin] (0*1.414,11.5*2.4495) -- (0*1.414,-8*2.4495);
\draw[very thin] (-13*1.414,4.333*2.4495) -- (13*1.414,-4.333*2.4495);
\draw[very thin] (-13*1.414,-4.333*2.4495) -- (13*1.414,4.333*2.4495);
\draw[very thin] (11.5*1.414,11.5*2.4495) -- (-8*1.414,-8*2.4495);
\draw[very thin] (8*1.414,-8*2.4495) -- (-11.5*1.414,11.5*2.4495);
\draw[fill=white] (7*1.414,11*2.4495) circle (0.35);
\draw[fill=black] (-4*1.414,-6*2.4495) circle (0.35);
\draw[fill=black] (-7*1.414,-5*2.4495) circle (0.35);
\end{tikzpicture}
  \caption{$\overline{\Pi}(12,0)$}
  \label{fig:g2222a}
\end{subfigure}%
\begin{subfigure}{.5\textwidth}
  \centering
\begin{tikzpicture}[scale=0.12]
\draw[fill=black!5,thick] (6*1.414,6*2.4495) --  (12*1.414,0*2.4495) -- (6*1.414,-6*2.4495) -- (-6*1.414,-6*2.4495) -- (-12*1.414,0*2.4495) -- (-6*1.414,6*2.4495) -- cycle;
\draw[pattern=crosshatch,thick] (-7.5*1.414,4.5*2.4495) -- (-7*1.414,5*2.4495) -- (7*1.414,5*2.4495) -- (8.5*1.414,3.5*2.4495)  -- (-1*1.414,-6*2.4495) -- (-6*1.414,-6*2.4495)  -- (-7.5*1.414,-4.5*2.4495)-- cycle;
\draw[fill=black!25,thick] (-7.5*1.414,4.5*2.4495) -- (-3*1.414,0*2.4495) -- (-7.5*1.414,-4.5*2.4495)  -- cycle;
\draw[fill=black!25,thick] (-7*1.414,5*2.4495) -- (7*1.414,5*2.4495) -- (6*1.414,4*2.4495) -- (-6*1.414,4*2.4495) -- cycle;
\draw[fill=black!25,thick] (8.5*1.414,3.5*2.4495) -- (1.5*1.414,3.5*2.4495) -- (-4.5*1.414,-2.5*2.4495) -- (-1*1.414,-6*2.4495) -- cycle;
\foreach \x in {-13,-12,...,13} {
	\foreach \y in {-13,-12,...,6} {
        \node at (\x*1.414,\y*2.4495) {$\cdot$};
    		}
		};
\foreach \x in {-13,-12,...,12} {
	\foreach \y in {-13,-12,...,6} {
        \node at (\x*1.414+0.5*1.414,1.225+\y*2.4495) {$\cdot$};
    		}
		};
\draw[thick,dashed] (-7.5*1.414,-13*2.4495) -- node[circle,draw=black,pos = 0.2,inner sep=1pt,fill=white,solid,thin] {\tiny$T_2$}  (-7.5*1.414,6.5*2.4495);
\draw[thick,dashed] (-12*1.414,5*2.4495) -- node[circle,draw=black,pos = 0.1,inner sep=1pt,fill=white,solid,thin] {\tiny$T_0$} (12.5*1.414,5*2.4495);
\draw[thick,dashed] (-8*1.414,-13*2.4495) --node[circle,draw=black,pos = 0.2,inner sep=1pt,fill=white,solid,thin] {\tiny$T_1$}  (11.5*1.414,6.5*2.4495);
\draw[fill=white] (0*1.414,0*2.4495) circle (0.35);
\draw[fill=black] (-4*1.414,-6*2.4495) circle (0.35);
\draw[fill=black] (-7*1.414,-5*2.4495) circle (0.35);
\end{tikzpicture}
  \caption{$\overline{\Pi}(12,0:14,4)$}
  \label{fig:g2222b}
\end{subfigure}
\caption{$(12,0)\otimes(14,4)\in\mathcal{C}(\mathfrak{g}_2,24)$}
\label{fig:g2222}
\end{figure}
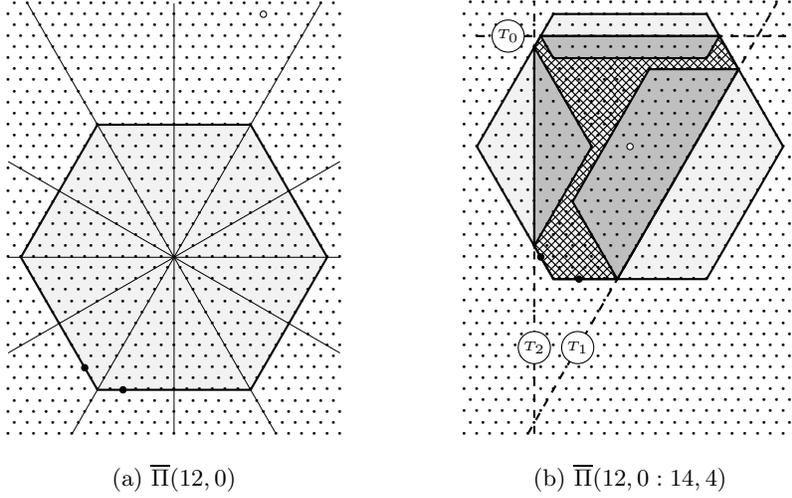

\begin{lem}\label{g2convex}
If $0\leq x<k/2$, $\dim'(x,0)\leq\dim'(s,t)$ for all $s,t\in\mathbb{Z}_{\geq0}$ such that $s+(3/2)t=x$.
\end{lem}
\begin{proof}
With $\kappa:=\pi/(3(k+4))$, define
\begin{align*}
f(t)&=\sin((3(t+1))\kappa) \\
g(t)&=\sin((x-(3/2)t+1)\kappa)\sin((x+(3/2)t+4)\kappa) \\
h(t)&=\sin((3x-(3/2)t+6)\kappa)\sin((3x+(3/2)t+9)\kappa)
\end{align*}
as real functions of $t\in[0,(2/3)x]$ so that
\begin{equation*}\dim'(x-(3/2)t,t)=\sin^{-6}(\kappa)\sin((2x+5)\kappa)f(t)g(t)h(t).\end{equation*}
Now we compute with $\alpha=x-(3/2)t+1$ and $\beta=x+(3/2)t+4$ for brevity,
\begin{align*}
&&g'(t)&=(3/2)\kappa(\cos(\beta\kappa)\sin(\alpha\kappa)-\sin(\beta\kappa)\cos(\alpha\kappa)) \\
&&&=-(3/2)\kappa\sin(3(t+1))\kappa) \\
\Rightarrow&&g''(t)&=-(9/2)\kappa^2\cos(3(t+1))\kappa).
\end{align*}
The derivatives of $g$ and $h$ coincide while the derivatives of $f$ are trivially computed, which (as in the proof of Lemma \ref{so5convex}) allows verification that $(d^2/dt^2)\dim'(x-(3/2)t,t)<0$ for $t\in[0,(2/3)x]$ using the product rule.  It then suffices to note $\dim'(x,0)\leq\dim'(0,x)$ to complete the proof.
\end{proof}

\par Lemma \ref{g2convex} along with Corollary \ref{goal} then implies
\begin{align}
&\dim'(x,0)+\dim'(x,0)\label{eq:g2eagles} \\
&\leq[x+4][3][3x+15][3x+18][x+7][2x+11]\label{eq:g2eagles2}\\
&<([x+1]+3)[3]([3x+6]+9)([3x+9]+9)([x+4]+3)([2x+5]+6)\label{eq:g2eagles3}
\end{align}
by applying Corollary \ref{quantumtrianglecor} to (\ref{eq:g2eagles2}).  Expanding the product in (\ref{eq:g2eagles3}) and subtracting the leading term (equal to $\dim'(x,0)$) from both sides of this equality yields
\begin{equation}
\dfrac{27}{64}(x+1)(x+2)(x+3)(x+4)(2x+5)<810(x+4)^2(x^2+8x+19)\label{g2eagles}
\end{equation}
by applying Lemma \ref{lowerquantumbound} to the factors on the left-hand side of (\ref{g2eagles}) and Corollary \ref{quantumtrianglecor} to the factors on the right-hand side of (\ref{g2eagles}).  The inequality in (\ref{g2eagles}) is true for $x\leq963$.  Moreover $\ell\leq1926$ and $m\leq963$, therefore $k\leq(1/3)(\ell^2+3\ell m+5\ell+3m^2+9m)-4$ is maximized within these bounds at $k\leq4023089$.




\section{Summary and further directions}\label{conclusion}

\par The level bounds presented for exceptional connected \'etale algebras in $\mathcal{C}(\mathfrak{so}_5,k)$ ($\approx2\times10^4$) and $\mathcal{C}(\mathfrak{g}_2,k)$ ($\approx2\times10^7$) should be sufficient in these cases to classify all connected \'etale algebras by modern computational methods (attributed to Gannon \cite[Section 1.5]{ocneanu}).  On the other hand these bounds are astronomical when compared to the highest level at which known exceptional connected \'etale algebras exist: $\mathcal{C}(\mathfrak{so}_5,12)$ and $\mathcal{C}(\mathfrak{g}_2,4)$.  The quantum inequalities found in Section \ref{sec:inequalities} are the main culprits for this disparity.  Coarse inequalities such as these are unavoidable in practice, as seen in Section \ref{soo3}, arguments relying on the definition of $[n]$ are necessarily verbose.  As the number of positive roots of $\mathfrak{g}$ increases, or equivalently the number of quantum factors in $\dim(\lambda)$ for $\lambda\in\Lambda_0$, arguments become infeasibly lengthy, and the level bounds given by the inequalities in Section \ref{sec:inequalities} will certainly grow outside of computational feasibility for a complete classification.  Producing tighter quantum inequalities could help curb both the length of arguments for level-finiteness for higher ranks and the growth rate of the level-bounds produced by these arguments.

\par Nonetheless the presented core argument for level-finiteness of exceptional connected \'etale algebras should be generalizable to all simple finite-dimensional complex Lie algebras $\mathfrak{g}$.  Given a connected \'etale algebra in $\mathcal{C}(\mathfrak{g},k)$ with a minimal (in an appropriately chosen sense) nontrivial summand $\gamma$, one should produce $\lambda,\mu_1,\mu_2$ such that $\mu_1\oplus\mu_2\subset\lambda\otimes\gamma$ and $\|\lambda-\mu_i\|$ is bounded by some fixed constant independent of $\gamma$ and $k$.  In rank 2 the choice of $\lambda,\mu_1,\mu_2$ was clear due to the geometric interpretation of the quantum Racah formula; higher rank Lie algebras may require an argument of mere existence due to the complexity of fusion rules, in which case the discovered level bounds will become even less tractable than those in this paper.

\bibliography{bib}

\begin{thebibliography}{10}

\bibitem{ap}
Henning~Haahr Andersen and Jan Paradowski.
\newblock Fusion categories arising from semisimple {L}ie algebras.
\newblock {\em Communications in Mathematical Physics}, (169(3)):563--588,
  1995.

\bibitem{BaKi}
Bojko Bakalov and Alexander~Kirillov Jr.
\newblock {\em Lectures on Tensor Categories and Modular Functors}, volume~21
  of {\em University Lecture Series}.
\newblock American Mathematical Society, 2001.

\bibitem{kawahigashi1999}
Jens B{\"o}ckenhauer, David~E. Evans, and Yasuyuki Kawahigashi.
\newblock On {$\alpha$}-induction, chiral generators and modular invariants for
  subfactors.
\newblock {\em Communications in Mathematical Physics}, 208(2):429--487, 1999.

\bibitem{kawahigashi2000}
Jens B{\"o}ckenhauer, David~E. Evans, and Yasuyuki Kawahigashi.
\newblock Chiral structure of modular invariants for subfactors.
\newblock {\em Communications in Mathematical Physics}, 210(3):733--784, 2000.

\bibitem{zuber}
Alfredo Cappelli, Claude Itzykson, and J.~B. Zuber.
\newblock The {A-D-E} classification of minimal and {$A_1^{(1)}$} conformal
  invariant theories.
\newblock {\em Communications in Mathematical Physics}, 113(1):1--26, 1987.

\bibitem{coq2}
Robert Coquereaux, Rochdi Rais, and El~Hassan Tahri.
\newblock Exceptional quantum subgroups for the rank two {L}ie algebras {$B_2$}
  and {$G_2$}.
\newblock {\em Journal of Mathematical Physics}, 51(092302), 2010.

\bibitem{coq1}
Robert Coquereaux and Gil Schieber.
\newblock Quantum symmetries for exceptional {$SU(4)$} modular invariants
  associated with conformal embeddings.
\newblock {\em Symmetry, Integrability and Geometry: Methods and Applications},
  5(044), 2009.

\bibitem{DMNO}
Alexei Davydov, Michael M{\"u}ger, Dmitri Nikshych, and Victor Ostrik.
\newblock The {W}itt group of nondegenerate braided fusion categories.
\newblock {\em Journal f{\"u}r die reine und angewandte Mathematik (Crelles
  Journal)}, (677):135--177, 2013.

\bibitem{DNO}
Alexei Davydov, Dmitri Nikshych, and Victor Ostrik.
\newblock On the structure of the {W}itt group of braided fusion categories.
\newblock {\em Selecta Mathematica}, (19):237--269, 2013.

\bibitem{DGNO}
Vladimir Drinfeld, Shlomo Gelaki, Dmitri Nikshych, and Victor Ostrik.
\newblock On braided fusion categories {I}.
\newblock {\em Selecta Mathematica}, (16):1--119, 2010.

\bibitem{sebas}
Sebas Eli{\"e}ns.
\newblock {Anyon condensation: Topological symmetry breaking phase transitions
  and commutative algebra objects in braided tensor categories}.
\newblock Master's thesis, Universiteit van Amsterdam, Netherlands, 2010.

\bibitem{tcat}
Pavel Etingof, Shlomo Gelaki., Dmitri Nikshych, and Victor Ostrik.
\newblock {\em Tensor Categories}.
\newblock Mathematical Surveys and Monographs. American Mathematical Society,
  2015.

\bibitem{fink}
Michael Finkelberg.
\newblock An equivalence of fusion categories.
\newblock {\em Geometric Functional Analysis}, 6:249--267, 1996.

\bibitem{rcft}
J{\"u}rgen Fuchs, Ingo Runkel, and Christoph Schweigert.
\newblock {TFT} construction of {RCFT} correlators {I}: partition functions.
\newblock {\em Nuclear Physics B}, 646(3):353 -- 497, 2002.

\bibitem{fuchs2017}
J{\"u}rgen Fuchs and Christoph Schweigert.
\newblock Consistent systems of correlators in non-semisimple conformal field
  theory.
\newblock {\em Advances in Mathematics}, 307:598 -- 639, 2017.

\bibitem{fuchs2013}
J{\"u}rgen Fuchs, Christoph Schweigert, and Alessandro Valentino.
\newblock Bicategories for boundary conditions and for surface defects in {3-d}
  {TFT}.
\newblock {\em Communications in Mathematical Physics}, 321(2):543--575, 2013.

\bibitem{gannon}
Terry Gannon.
\newblock The classification of {$SU(3)$} modular invariants revisited.
\newblock {\em Annales de {l'I.H.P.} Physique Theorique}, 65(1):15--55, 1996.

\bibitem{pinhasnoah}
Pinhas Grossman and Noah Snyder.
\newblock Quantum subgroups of the {H}aagerup fusion categories.
\newblock {\em Communications in Mathematical Physics}, 311(3):617--643, 2012.

\bibitem{huang2015}
Yi-Zhi Huang, Alexander Kirillov, and James Lepowsky.
\newblock Braided tensor categories and extensions of vertex operator algebras.
\newblock {\em Communications in Mathematical Physics}, 337(3):1143--1159,
  2015.

\bibitem{hump}
James~E. Humphreys.
\newblock {\em Introduction to Lie Algebras and Representation Theory}.
\newblock Graduate Texts in Mathematics. Springer, 1972.

\bibitem{KiO}
Alexander Kirillov and Victor Ostrik.
\newblock On a {$q$}-analogue of the {McKay} correspondence and the {ADE}
  classification of {$\mathfrak{sl}_2$} conformal field theories.
\newblock {\em Advances in Mathematics}, 171(2):183--227, 2002.

\bibitem{knuth}
Donald~E. Knuth.
\newblock Johann {F}aulhaber and sums of powers.
\newblock {\em Mathematics of Computation}, 61(203):277--294, 1993.

\bibitem{kong}
Liang Kong.
\newblock Anyon condensation and tensor categories.
\newblock {\em Nuclear Physics B}, 886:436--482, 2014.

\bibitem{ocneanu}
Adrian Ocneanu.
\newblock {\em The classification of subgroups of quantum {$SU(n)$}}, volume
  294 of {\em Contemporary Mathematics}.
\newblock American Mathematical Society, 2002.

\bibitem{Ostrikdouble}
Victor Ostrik.
\newblock Module categories over the {Drinfeld} double of a finite group.
\newblock {\em International Mathematics Research Notices},
  2003(27):1507--1520, 2003.

\bibitem{Ostrik2003}
Victor Ostrik.
\newblock Module categories, weak {Hopf} algebras and modular invariants.
\newblock {\em Transformation Groups}, 8(2):177--206, 2003.

\bibitem{ENO}
Dmitri~Nikshych Pavel~Etingof and Victor Ostrik.
\newblock On fusion categories.
\newblock {\em Annals of Mathematics}, 162(2):581--642, 2005.

\bibitem{Sawin03}
Stephen Sawin.
\newblock Quantum groups at roots of unity and modularity.
\newblock {\em Journal of Knot Theory and Its Ramifications}, (15):1245--1277,
  2006.

\bibitem{schopieray2017}
Andrew Schopieray.
\newblock Classification of {$\mathfrak{sl}_3$} relations in the {W}itt group
  of nondegenerate braided fusion categories.
\newblock {\em Communications in Mathematical Physics}, 353(3):1103--1127,
  2017.

\end{thebibliography}
\bibliographystyle{plain}

\end{document}